\documentclass{article}

\usepackage[margin=3cm]{geometry}
\usepackage{amsmath,amssymb,amsbsy,graphicx,bbm}
\usepackage{amsthm}
\usepackage{natbib}
\usepackage[T1]{fontenc}
\usepackage{setspace}
\usepackage{enumitem}
\usepackage[pdftex,dvipsnames]{xcolor}
\usepackage{authblk}
\usepackage{subcaption}
\usepackage{lmodern}
\usepackage{dsfont}
\usepackage{algorithm}
\usepackage{algorithmic}
\usepackage[colorlinks,allcolors=black]{hyperref}

\theoremstyle{plain}
\newtheorem{theorem}{Theorem}[section]
\newtheorem{lemma}[theorem]{Lemma}
\newtheorem{proposition}[theorem]{Proposition}
\newtheorem{corollary}[theorem]{Corollary}

\theoremstyle{remark}
\newtheorem{definition}[theorem]{Definition}
\newtheorem*{example}{Example}

\renewcommand{\P}{\mathbb{P}}
\newcommand{\R}{\mathbb{R}}
\newcommand{\N}{\mathbb{N}}
\newcommand{\Z}{\mathbb{Z}}

\newcommand{\I}{\mathbb{I}}

\newcommand{\mF}{\mathcal{F}}

\newcommand{\mN}{\mathcal{N}}
\newcommand{\mS}{\mathcal{S}}

\newcommand{\Argmin}{\operatorname*{\text{\rm Arg}\min}}
\newcommand{\prox}{\text{\rm prox}}
\newcommand{\ind}{\mathds{1}}

\newcommand{\eps}{\varepsilon}
\newcommand{\row}{{\rm row}}
\newcommand{\col}{{\rm col}}
\newcommand{\rk}{{\rm rk}}
\newcommand{\rank}{{\rm rank}}
\newcommand{\codim}{{\rm codim}}
\newcommand{\defect}{{\rm def}}
\newcommand{\pen}{{\rm pen}}
\newcommand{\penequiv}{\overset{\pen}{\sim}}
\newcommand{\supequiv}{\overset{\|.\|_\infty}{\sim}}
\newcommand{\lin}{{\rm lin}}
\newcommand{\conv}{{\rm conv}}
\newcommand{\cone}{{\rm cone}}

\newcommand{\sign}{{\rm sign}}
\newcommand{\patt}{{\rm patt}}
\newcommand{\sure}{{\rm sure}}
\newcommand{\pattSLOPE}{{\rm patt}_{\text{\rm slope}}}
\newcommand{\pattTV}{{\rm patt}_{\text{\rm tv}}}
\newcommand{\pattTF}{{\rm patt}_{\text{\rm tf}}}

\newcommand{\relint}{{\rm ri}}
\newcommand{\aff}{{\rm aff}}

\newcommand{\betaLASSO}{\hat\beta^{\rm LASSO}}
\newcommand{\betaLASSOtau}{\hat\beta^{{\rm LASSO},\tau}}
\newcommand{\betaSLOPE}{\hat\beta^{\rm SLOPE}}
\newcommand{\betathres}{\hat\beta^{\rm thr}}

\newcommand{\Spen}{S_{X,{\rm pen}}}
\newcommand{\SpenL}{S_{X,\lambda {\rm pen}}}
\newcommand{\SpenLr}{S_{X,\lambda^{(r)} {\rm pen}}}
\newcommand{\SpenZL}{S_{Z,\lambda\pen}}

\newcommand{\SsupL}{S_{X,\lambda\|.\|_\infty}}

\newcommand{\SlassoL}{S_{X,\lambda\|.\|_1}}

\newcommand{\dlasso}{\partial_{\|.\|_1}}
\newcommand{\dslope}{\partial_{\|.\|_w}}
\newcommand{\dpen}{\partial_{{\rm pen}}}
\newcommand{\dgen}{\partial_{\|D.\|_1}}
\newcommand{\dsup}{\partial_{\|.\|_\infty}}
\newcommand{\dgtv}{\partial_{\|D^{\rm tv}.\|_1}}
\newcommand{\dgtf}{\partial_{\|D^{\rm tf}.\|_1}}

\renewcommand{\emptyset}{\varnothing}
\newcommand{\dphi}{\partial_\phi}
\newcommand{\Dtv}{D^{\rm tv}}
\newcommand{\Dtf}{D^{\rm tf}}

\usepackage{todonotes}
\usepackage{marginnote}

\newcommand{\lefttodo}[2][]{{%
 \let\marginpar\marginnote
 \reversemarginpar
 \renewcommand{\baselinestretch}{0.8}%
 \todo[#1]{#2}}}

\begin{document}

\onehalfspacing 

\author[1]{Piotr Graczyk}   \affil[1]{Université d'Angers, Angers, France}
\author[2]{Ulrike Schneider} 
 \affil[2]{TU Wien, Vienna, Austria}
\author[1,3]{Tomasz Skalski} \affil[3]{Politechnika Wroc\l{}awska, Wroc\l{}aw, Poland}
\author[4]{Patrick Tardivel} \affil[4]{Université Bourgogne Europe, CNRS, IMB UMR 5584, Dijon, France}

\title{A Unified Framework for Pattern Recovery in Penalized and
Thresholded Estimation and its Geometry}

\date{}

\maketitle

\begin{abstract}
We consider the framework of penalized estimation where the penalty
term is given by a real-valued polyhedral gauge, which encompasses
methods such as LASSO, generalized LASSO, SLOPE, OSCAR, PACS and
others. Each of these estimators is defined through an optimization
problem and can uncover a different structure or ``pattern'' of the
unknown parameter vector. We define a novel and general notion of
patterns based on subdifferentials and formalize an approach to
measure pattern complexity. For pattern recovery, we provide a minimal
condition for a particular pattern to be detected by the procedure
with positive probability, the so-called accessibility condition.
Using our approach, we also introduce the stronger noiseless recovery
condition. For the LASSO, it is well known that the irrepresentability
condition is necessary for pattern recovery with probability larger
than $1/2$ and we show that the noiseless recovery plays exactly the
same role in our general framework, thereby unifying and extending the
irrepresentability condition to a broad class of penalized estimators.
We also show that the noiseless recovery condition can be relaxed when
turning to so-called thresholded penalized estimators: we prove that
the necessary condition of accessibility is already sufficient for
sure pattern recovery by thresholded penalized estimation provided
that the noise is small enough. Throughout the article, we demonstrate
how our findings can be interpreted through a geometrical lens.
\end{abstract}

\smallskip

\medskip


\smallskip


\newpage

\section{Introduction} \label{sec:intro}

Consider the linear regression model
$$ 
y = X\beta + \eps,
$$
where $X \in \R^{n \times p}$ is a design matrix, $\eps \in \R^n$
represents random noise and $\beta \in \R^p$ is the vector of unknown
regression coefficients. Penalized estimation of $\beta$ has been
studied extensively in the literature, of particular interest the case
where the penalization is polyhedral so that the estimator may detect
particular features of $\beta$. Depending on the penalty term,
different characteristics can be unveiled by the procedure. The most
prominent example is of course the LASSO \citep{Tibshirani96} with its
ability to perform model selection, i.e., potentially uncovering zero
components of $\beta$. In addition to this sparsity property, the
fused LASSO \citep{TibshiraniEtAl05} may set adjacent components to be
equal. Using the supremum norm promotes clustering of components that
are maximal in absolute value \citep{JegouEtAl20}. SLOPE
\citep{BogdanEtAl15} as well as OSCAR \citep{BondellReich08} display
further clustering phenomena where certain components may be equal in
absolute value -- to name just a few. When looking closer at these
phenomena under a geometric lens, \cite{SchneiderTardivel22} show that
for the LASSO, the naturally arising pattern structure not only
carries information about zero components, but also about the signs of
the non-zero components. This natural pattern structure of the LASSO
has appeared many times in the literature, such as in the conditioning
event in the selective inference approach of \cite{LeeEtAl16}, in the
so-called sign-consistency of the LASSO \citep[see
e.g.][]{ZhaoYu06,TardivelBogdan22}, sign accessibility of the LASSO
\citep{SepehriHarris17}, and in the solution path of the LASSO
\citep{MairalYu2012}. For SLOPE, the natural pattern structure
describes not only signs (zero components as well as the signs of
non-zero components) and clustering (components may be equal in
absolute value, a well-known phenomenon for this estimator), but in
addition conveys information about the ordering of the coefficients.
This pattern structure of SLOPE has also appeared in the literature,
such as in the so-called pattern-consistency of  SLOPE
\citep{BogdanEtAl25}, pattern accessibility of SLOPE
\citep{SchneiderTardivel22} and in the solution path of SLOPE
\citep{DupuisTardivel24}.

In this article, we provide a general approach that allows to
characterize the pattern structure naturally arising for a particular
penalized estimation method. We do so by introducing the notion of
patterns inherent to a method as equivalence classes of elements in
$\R^p$ exhibiting the same subdifferential with respect to the
penalizing term. We assume that the penalty term is given by a
polyhedral gauge, a concept slightly more general than a polyhedral
norm which allows to also treat methods such as the generalized LASSO
\citep{TibshiraniTaylor11}.

We show that the pattern equivalence classes coincide with the
relative interiors of the normal cones of the polytope $B^*$, where
$B^*$ is the subdifferential of the penalizing gauge at zero and that
the correspondence between equivalence classes and faces of $B^*$ is a
bijection. Moreover, the linear span of a pattern equivalence class is
a model subspace as defined in \cite{VaiterEtAl15,VaiterEtAl18}. The
partition into pattern classes and the partition into faces of $B^*$
can be viewed as a the natural stratifications of the so-called
mirror-stratifiable penalizing gauge, as treated in
\cite{FadiliEtAl18}. We also introduce the concept of complexity of a
pattern, defined to be the dimension of the linear span of the
corresponding equivalence class, and prove that this complexity
measure coincides with the codimension of the associated face of
$B^*$.

Given this general notion of patterns, we turn to the question of when
an estimation procedure may recover a specific pattern. A minimal
condition is the so-called accessibility condition of a pattern of
$\beta$ which gives equivalent criteria for the existence of point $y
\in \R^n$ such that the resulting estimator exhibits the pattern under
consideration. We express this criterion both in an analytic manner
and through a geometric criterion involving how the row span of $X$
intersects the polytope $B^*$. This extends the geometric condition
given for LASSO and SLOPE in \cite{SchneiderTardivel22} to the general
framework of gauge-penalized estimation. Note that a different
approach for an accessibility criterion for the LASSO under a
uniqueness assumption was also considered in \cite{SepehriHarris17}.
Under uniqueness, we prove that this minimal condition already ensures
pattern detection with positive probability, provided that the
response vector follows a continuous distribution on $\R^n$.

A stronger condition is given by the noiseless recovery condition,
where the estimator determined by the noiseless signal $y = X\beta$ is
required to possess the same pattern as $\beta$ for some value of the
tuning parameter. This condition can be proven to be equivalent to the
irrepresentability condition in case of the LASSO \citep[see
e.g.][]{BuehlmannVdGeer11} which is a necessary condition for pattern
recovery with probability of at least $1/2$ \citep{Wainwright09}. In
fact, the noiseless recovery condition is shown to play exactly the
same role as the irrepresentability condition in the general
gauge-penalized estimation framework: it is a necessary condition for
pattern recovery with probability of at least $1/2$ and allows to
unify and extend the concept of an irrepresentability condition to
entire class of gauge-penalized estimators. We also provide a
geometric criterion for this condition and discuss the ``gap'' between
accessibility and noiseless recovery. This geometric approach allows
to observe that accessibility is easier to fulfill for less complex
patterns whereas the gap between accessibility and noiseless recovery
becomes smaller when complexity increases.

It is known that the condition under which the support of the LASSO
contains the support of the regression parameter is weaker  than the
irrepresentability condition \citep{FadiliEtAl19}. More precisely,
sign recovery by thresholded LASSO -- where small non-zero components
may be set to zero additionally to existing zeros -- improves sign
recovery by LASSO \citep{TardivelBogdan22}. Inspired by these facts,
we define a general concept of thresholded estimators that alter the
penalized estimator by in some sense moving to a nested less complex
pattern. We show that for this thresholded penalized estimation, the
noiseless recovery condition which is necessary for pattern recovery
without thresholding, the much weaker accessibility condition is
already sufficient for sure pattern recovery under a uniqueness
assumption, provided that the noise is small enough.

For completeness, we also extend the necessary and sufficient
condition for uniform uniqueness from \cite{SchneiderTardivel22} to
gauge-penalized estimation which again relies on the connection
between patterns and the faces of $B^*$ and essentially shows that
uniqueness occurs if no pattern of complexity exceeding the rank of
$X$ is accessible. Numerous articles in the literature have  examined
the uniqueness of solutions in penalized estimation, particularly in
scenarios where $ X \in \R^{n \times p} $ and $ y \in \R^n $ are fixed
\citep{FadiliEtAl25,Gilbert17,MousaviShen19}. However, in statistics,
$y$ typically represents a random variable thus a stronger concept of
uniqueness, termed uniform uniqueness, which holds for all $ y \in
\R^n $, is pertinent
\citep{Tibshirani13,AliTibshirani19,EwaldSchneider20,SchneiderTardivel22}.

Finally, we illustrate some pattern recovery properties with numerical
experiments.

The paper is organized as follows. In Section~\ref{sec:setting}, we
introduce the given setting and notation. Section~\ref{sec:patterns}
treats defining and illustrating pattern structures. Pattern recovery
by penalized estimation is investigated in Section~\ref{sec:recovery},
whereas we turn to pattern recovery by thresholded penalized
estimation in Section~\ref{sec:recovery-thres}. Uniform uniqueness is
proven in Section~\ref{sec:unique} and Section~\ref{sec:num-exp} gives
some numerical illustrations. Section~\ref{sec:concl} concludes. All
proofs are relegated to Appendix~\ref{app:proofs}, before which
Appendix~\ref{app:polytopes} provides some definitions and results on
polytopes and gauges. Finally, Appendix~\ref{app:add-results} contains
additional results referred to throughout, including a result on
solution existence of the optimization problem treated in the article.

\section{Setting and Notation} \label{sec:setting}

The optimization problem we consider throughout the article is the
gauge-penalized least-squares problem described in the following. Let
$X \in \R^{n \times p}$ be completely arbitrary. Given $y \in \R^n$
and $\lambda > 0$, we define the set $\SpenL(y)$ of minimizers to be
given by
\begin{equation} \label{eq:pen_problem}
\SpenL(y) = \Argmin_{b \in \R^p} \frac{1}{2} \|y - Xb\|_2^2 + \lambda \pen(b),
\end{equation}
where ``$\pen$'' is a real-valued polyhedral gauge and $\|.\|_2$
denotes the Euclidean norm. A gauge is any non-negative and positively
homogeneous convex function that vanishes at $0$, and it is polyhedral
if its unit ball is given by a (possibly unbounded) polyhedron. A
polyhedral gauge $b \in \R^p \mapsto \pen(b) \in [0,\infty)$ can
always be written as the maximum of finitely many linear functions
\citep{Rockafellar97,MousaviShen19}, so that we can assume that
$$
\pen(b) = \max\{u_1'b,\dots,u_k'b\}, \text{ \rm for some }
u_1,\dots,u_k\in \R^p \text{ \rm with } u_1 = 0.
$$
Note that a polyhedral gauge  whose unit ball $B = \{b \in \R^p:
\pen(x) \leq 1\}$ is a bounded and symmetric polyhedron is in fact a
polyhedral norm. Examples of polyhedral norms and gauges are discussed
in more detail in Section~\ref{sec:patterns}. For our geometric
considerations, a central object of study will be the polytope $B^*$
defined as
$$
B^* = \conv(u_1,\dots,u_k),
$$
where $\conv(.)$ denotes the convex hull. In case $\pen$ is a norm,
$B^*$ coincides with the unit ball of the dual norm. 
The optimization problem in \eqref{eq:pen_problem} always possesses a
solution, as we show in Proposition~\ref{prop:existence} in
Appendix~\ref{app:add-results}\footnote{The existence of a minimizer is
clear when $\pen$ is a norm. For the special case of the generalized
LASSO (in which $\pen$ is not a norm), existence is shown in
\cite{AliTibshirani19} or \cite{DupuisVaiter23}. However, these
proofs cannot be generalized to arbitrary polyhedral gauges.}, but it
does not have to be unique. We treat uniqueness by giving a necessary
and sufficient condition in Section~\ref{sec:unique}.

\medskip

The following additional notation will be used throughout the article.
By $[p]$, we denote the set $\{1,\dots,p\}$. For a set $I \subseteq
[p]$, the symbol $I^c$ denotes its complement $I^c = [p]\setminus I$.
Given a matrix $X$ and an index set $I$, $X_I$ is the matrix with
columns corresponding to indices in $I$ only, with analogous notation
for a vector $b$, so that $b_I$ denotes the vector with components
with indices in $I$ only. The column of $X$ is $\col(X)$ and we define
the row space of $X$ to be $\row(X) = \col(X')$. We refer to the rank
of $X$ as $\rk(X)$. For a set $S \subseteq \R^p$, $\lin(S)$ is the
linear span of $S$, i.e., the smallest vector space containing $S$ and
$\aff(S)$ is the affine hull of $S$, i.e., the smallest affine space
containing $S$, whereas $\vec\aff(S)$ refers to the vector space
parallel to $\aff(S)$ given by $\{u - s: u \in \aff(S)\}$ for a fixed,
but arbitrary $s \in \aff(S)$. The relative interior of $S$ is denoted
by $\relint(S)$. The symbol $V^\perp$ is used for the orthogonal
complement of the vector space $V$ and $\ind(.)$ stands for the
indicator function. The sign of a number $a$ is $\sign(a) = \ind\{a
\geq 0\} - \ind\{a \leq 0\} \in \{-1,0,1\}$. For a vector $b$,
$\sign(b)$ is a vector of the same dimension with the $\sign$-function
applied componentwise. Moreover, for a convex function $\phi : \R^p
\to \R$, a vector $s \in \R^p$ is a \emph{subgradient of $\phi$ at
$\beta \in \R^p$} if
$$
\phi(b) \geq \phi(\beta) + s'(b-\beta) \;\; \forall b \in \R^p.
$$
The convex, non-empty set of all subgradients of $\phi$ at $\beta$ is
called the \emph{subdifferential of $\phi$ at $\beta$}, denoted by
$\partial_\phi(\beta)$. For a closed and convex set $K \subseteq \R^p$
and $\beta \in K$, the \emph{normal cone of $K$ at $\beta$} is given
by
$$
N_K(\beta) = \{s \in \R^p: s'(b-\beta) \leq 0 \; \forall b \in K\},
$$
see e.g.\ \citet[p.65]{HiriartLemarechal01}. Note that, by definition,
a normal cone contains $0$. Finally, $\mN(\mu,\sigma^2)$ denotes a
univariate normal distribution with mean $\mu$ and variance
$\sigma^2$.

\section{The Notion of Patterns, Pattern Complexity and Their
Connection to the Faces of $B^*$} \label{sec:patterns}
 
For a gauge-penalized estimation method, we determine its canonical
pattern structure and pattern complexity through the following
definition.

\begin{definition}[Pattern equivalence class] \label{def:equiv-rel}
Let $\pen$ be a real-valued polyhedral gauge on $\R^p$. We say that
$\beta$ and $\tilde\beta \in \R^p$ have the same \emph{pattern with
respect to $\pen$} if
$$
\dpen(\beta) = \dpen(\tilde\beta),
$$
i.e., if their subdifferentials of $\pen$ coincide. We then write
$\beta \penequiv \tilde\beta$. The set of all elements of $\R^p$
sharing the same pattern as $\beta$ is called the pattern equivalence
class $C_\beta$. Furthermore, we define the \emph{complexity} of the
pattern of $\beta$ to be the dimension of $\lin(C_\beta)$.
\end{definition}

There is an intrinsic connection between the patterns with respect to
$\pen$ and the faces of the polytope $B^*$ determined by $\pen$: it
can be shown that $B^* = \dpen(0)$ and that the other subdifferentials
$\dpen(\beta)$ with $\beta \neq 0$ make up the faces of $B^*$, see
Lemma~\ref{lem:subdiff-faces} in Appendix~\ref{app:polytopes} for
details. Therefore, there is a one-to-one relationship between the
pattern equivalence classes $C_\beta$ and the non-empty faces of $B^*$
(note that formally $B^*$ itself is also a face). In fact, this
relationship can be made fully concrete in the following theorem,
showing a pattern equivalence class $C_\beta$ in fact \emph{equals}
the (relative interior of the) normal cone of a face. To better
understand the corresponding statement in the theorem below, also note
that any point in the relative interior of a face of a polytope will
give rise to the same normal cone and that a normal cone always
``sits'' at the origin (in the sense that either 0 is the unique
extremal point of the cone or that the cone is ``centered'' there).
Figures~\ref{fig:pattern_LASSO}-\ref{fig:pattern_TV} are meant to
illustrate this further.

\begin{theorem} \label{thm:patt-cones}
Let $\pen$ be a real-valued polyhedral gauge on $\R^p$ and let $\beta
\in \R^p$. Then $C_\beta = \relint(N_{B^*}(s))$ where $s$ is an
arbitrary element of $\relint(\dpen(\beta))$ and $\lin(C_\beta) =
\vec\aff(\dpen(\beta))^\perp$.
\end{theorem}
Note that the second part of the above theorem, which states that the
linear span of a pattern equivalence class equals the orthogonal
complement of the vectorized affine span of the corresponding
subdifferential or face, also demonstrates that the measure of
complexity of the pattern of $\beta$ introduced in
Definition~\ref{def:equiv-rel} coincides with the codimension of the
face $\dpen(\beta)$ of $B^*$ which is given by $p -
\dim(\dpen(\beta))$ as summarized in the corollary below.\footnote{The
dimension of a face is defined as the dimension of its affine hull,
see Appendix~\ref{app:polytopes} for details.} This quantity is also
relevant for uniform uniqueness characterized in
Theorem~\ref{thm:unique-gauge}. Additionally, the above statement
proves that $\lin(C_\beta)$ matches the notion of model subspace as
defined in \cite{VaiterEtAl15, VaiterEtAl18}.

As it is known that the relative interiors of the normal cones of a
polytope form a partition $\R^p$ \citep[see][p. 17, Theorem
4.13]{Ewald96}, the first part of Theorem~\ref{thm:patt-cones} shows
that this partition with respect to $B^*$ is the same as partitioning
the space by the pattern equivalence classes $C_\beta$. Moreover,
Theorem~\ref{thm:patt-cones} provides a geometrical construction --
via normal cones and faces -- of the partitions in the primal ($\R^p$)
and dual ($B^*$) spaces naturally induced by the mirror-stratifiable
polyhedral gauge \cite[Proposition 2]{FadiliEtAl18}. In this context,
the above theorem can also be viewed to provide the correspondence
operators of the stratifications.

\begin{corollary} \label{cor:compl-codim}
Let $\pen$ be a real-valued polyhedral gauge on $\R^p$ and let $\beta
\in \R^p$. Then the complexity of the pattern of $\beta$ with respect
to $\pen$ is given by the codimension of $\dpen(\beta)$.
\end{corollary}

The actual pattern structure for a given penalty term has to be
understood on a case-by-case basis. We illustrate the notion of
patterns and their complexity as well as the above theorem for several
examples of gauges in the following.

\begin{example}[Different penalizations and their patterns] \label{ex:patterns} \leavevmode

\begin{description}


\item[{\bf $\ell_1$-norm:}] The subdifferential of the $\ell_1$-norm at $0$ is given
by $B^* = \dlasso(0) = [-1,1]^p$. The pattern of $\beta \in \R^p$ can
be represented by its sign vector, $\sign(\beta) \in \{-1,0,1\}^p$ with
$$
\sign(\beta) = (\sign(\beta_1),\dots,\sign(\beta_p))'.
$$
Indeed, the subdifferentials $\dlasso(.)$ at two points in $\R^p$ will
be the same if and only if their sign vectors coincide so that
$C_\beta = \{b \in \R^p : \sign(b) = \sign(\beta)\}$ and the pattern
structure of the LASSO carries not only information about zero
components, but also the signs of the non-zero coefficients. Note that
the complexity of the LASSO pattern of $\beta$, which coincides with
the codimension of $\dlasso(\beta)$ by
Corollary~\ref{cor:compl-codim}, is given by $\|\sign(\beta)\|_1$, the
number of non-null components of $\beta$. See also
Figure~\ref{fig:pattern_LASSO}.

\begin{figure}[h!]

\centering

\begin{tikzpicture}[scale=0.7]

\fill[color=blue,fill opacity=0.7] (1,1) -- (1,-1) -- (-1,-1) -- (-1,1)-- (1,1);
\fill[color=red,fill opacity=0.1] (1,1) -- (1,4) -- (4,4) -- (4,1) -- (1,1);
\fill[color=red,fill opacity=0.1] (1,-1) -- (1,-4) -- (4,-4) -- (4,-1) -- (1,-1);
\fill[color=red,fill opacity=0.1] (-1,1) -- (-1,4) -- (-4,4) -- (-4,1) -- (-1,1);
\fill[color=red,fill opacity=0.1] (-1,-1) -- (-1,-4) -- (-4,-4) -- (-4,-1) -- (-1,-1);

\draw[color=green] (1,0)-- (4,0);
\draw[color=green] (-1,0)-- (-4,0);
\draw[color=green] (0,1)-- (0,4);
\draw[color=green] (0,-1)-- (0,-4);

\fill[color=red,fill opacity=0.1] (14,4) -- (14,-4) -- (6,-4) -- (6,4)-- (14,4);

\draw[color=green] (6,0)-- (14,0);
\draw[color=green] (10,-4)-- (10,4);

\draw (10,0) node[color=blue,fill opacity=0.7,circle,fill,inner sep=0pt,minimum size=5pt]{};

\draw (13.4,0.5) node[color=green]{$(1,0)'$};
\draw (9.3,3.5) node[color=green]{$(0,1)'$};
\draw (6.6,-0.5) node[color=green]{$(-1,0)'$};
\draw (10.9,-3.5) node[color=green]{$(0,-1)'$};

\draw (13.2,3.5) node[color=red]{$(1,1)'$};
\draw (7,-3.5) node[color=red]{$(-1,-1)'$};
\draw (6.9,3.5) node[color=red]{$(-1,1)'$};
\draw (13.2,-3.5) node[color=red]{$(1,-1)'$};

\draw (10.6,0.5) node[color=blue,fill opacity=0.7]{$(0,0)'$};

\end{tikzpicture}
 
\caption{\label{fig:pattern_LASSO} Pattern equivalence classes for the
LASSO in $p=2$ dimensions: On the left, the blue polytope is $B^* =
\dlasso(0) = \conv\{\pm(1,1)',\pm(1,-1)'\}$, together with the shifted
normal cones of the faces of $B^*$  in pink and green. To visualize
the correspondence between a face and a normal cone, the origin is
translated to the middle of the face. The picture on the right
provides the relative interior of the normal cones (containing the
origin on their boundary), coinciding with the pattern equivalence
classes $C_\beta = \{b \in \R^p: \sign(b) = \sign(\beta)\}$ for the
patterns $\sign(\beta) \in \{(0,0)',\pm(0,1)',\pm(1,
0)',\pm(1,1)',\pm(1,-1)'\}$.}
\end{figure}
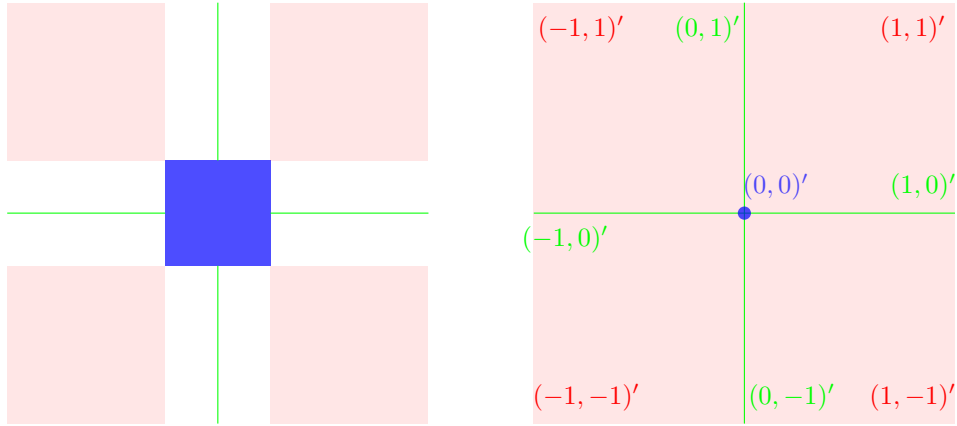


\item[{\bf Sorted-$\ell_1$-norm:}] For $b \in \R^p$, the
sorted-$\ell_1$-norm is defined as $\|b\|_w = \sum_{j=1}^p
w_j|b|_{(j)},$ where $|b|_{(1)} \geq \dots \geq |b|_{(p)}$ and $w_1
\geq \dots \geq w_p \geq 0$ with $w_1 > 0$ are pre-defined weights. It
can be shown that $B^*=\dslope(0) =
\conv\{(\pm w_{\pi(1)},\dots,\pm w_{\pi(p)})' : \pi \in \mS_p\}$ with $\mS_p$
denoting the set of all permutations on $[p]$. The polytope $B^*$ is
the so-called signed permutahedron, see \cite{NegrinhoMartins14} and
\cite{SchneiderTardivel22} for details. The SLOPE pattern of $\beta
\in \R^p$ is represented by $\pattSLOPE(\beta) \in \Z^p$ with each
component given by
$$
\pattSLOPE(\beta)_j = \sign(\beta_j)\,\mathrm{rank}(|\beta|)_j,
$$
where $\rank(|\beta|)_j \in \{0,1,\dots,m\}$ with $m$ the number of
(distinct) non-zero values in $\{|\beta_1|,\dots,|\beta_p|\}$ is
defined as follows: $\mathrm{rank}(|\beta|)_j = 0$ if $\beta_j = 0$,
$\rank(|\beta|)_j > 0$ if $|\beta_j| > 0$ and $\rank(|\beta|)_i <
\rank(|\beta|)_j$ if $|\beta_i|< |\beta_j|$, as can be learned in
\cite{SchneiderTardivel22}. For example, the SLOPE pattern of $\beta =
(3.1,-1.2,0.5,0,1.2,-3.1)'$ is given by $\pattSLOPE(\beta) =
(3,-2,1,0,2,-3)'$. Indeed, if $w \in \R^p$ satisfies $w_1 > \dots >
w_p > 0$, the subdifferentials $\dslope(.)$ at two points in $\R^p$
will be the same if and only if their SLOPE patterns coincide so that
$C_\beta = \{b \in \R^p : \pattSLOPE(b) = \pattSLOPE(\beta)\}$. This
shows that the SLOPE patterns do not only carry information about
zeros, signs and clustering, but also about the order of the clusters.
SLOPE patterns are also treated in \cite{HejnyEtAl23TR}. The
complexity of the SLOPE pattern of $\beta$ is given by
$\|\pattSLOPE(\beta)\|_\infty$, the number of non-zero clusters in
$\beta$, see \cite{SchneiderTardivel22}. See also
Figure~\ref{fig:pattern_SLOPE}.

\begin{figure}[h!]

\centering

\begin{tikzpicture}[scale=0.7]

\fill[color=blue,fill opacity=0.7] (2,1) -- (2,-1) -- (1,-2) -- (-1,-2)--(-2,-1) -- (-2,1) --
(-1,2)-- (1,2) -- (2,1);
\fill[color=red,fill opacity=0.1] (1,2) -- (1,4) -- (3,4) -- (1,2);
\fill[color=red,fill opacity=0.1] (-1,2) -- (-1,4) -- (-3,4) --(-1,2);
\fill[color=red,fill opacity=0.1] (1,-2) -- (1,-4) -- (3,-4) --(1,-2);
\fill[color=red,fill opacity=0.1] (-1,-2) -- (-1,-4) -- (-3,-4) --(-1,-2);
\fill[color=red,fill opacity=0.1] (2,1) -- (4,1) -- (4,3) -- (2,1);
\fill[color=red,fill opacity=0.1] (-2,1) -- (-4,1) -- (-4,3) -- (-2,1);
\fill[color=red,fill opacity=0.1] (2,-1) -- (4,-1) -- (4,-3) -- (2,-1);
\fill[color=red,fill opacity=0.1] (-2,-1) -- (-4,-1) -- (-4,-3) -- (-2,-1);

\draw[color=green] (1.5,1.5)-- (4,4);
\draw[color=green] (-1.5,-1.5)-- (-4,-4);
\draw[color=green] (1.5,-1.5)-- (4,-4);
\draw[color=green] (-1.5,1.5)-- (-4,4);
\draw[color=green] (2,0)-- (4,0);
\draw[color=green] (-2,0)-- (-4,0);
\draw[color=green] (0,-2)-- (0,-4);
\draw[color=green] (0,2)-- (0,4);

\fill[color=red,fill opacity=0.1] (6,-4) -- (14,-4) -- (14,4) -- (6,4) -- (6,-4);

\draw[color=green] (6,-4)-- (14,4);
\draw[color=green] (6,4)-- (14,-4);
\draw[color=green] (6,0)-- (14,0);
\draw[color=green] (10,-4)-- (10,4);
\draw (10,0) node[color=blue,fill opacity=0.7,circle,fill,inner sep=0pt,minimum size=5pt]{};
\draw (10.9,0.25) node[color=blue,fill opacity=0.7]{$(0,0)'$};
\draw (13.2,0.3) node[color=green]{$(1,0)'$};
\draw (6.8,-0.3) node[color=green]{$(-1,0)'$};
\draw (9.2,3.7) node[color=green]{$(0,1)'$};
\draw (10.8,-3.7) node[color=green]{$(0,-1)'$};
\draw (13,3.7) node[color=green]{$(1,1)'$};
\draw (7.2,-3.7) node[color=green]{$(-1,-1)'$};
\draw (7.2,3.7) node[color=green]{$(-1,1)'$};
\draw (12.8,-3.7) node[color=green]{$(1,-1)'$};
\draw (12.5,1) node[color=red]{$(2,1)'$};
\draw (11,2.5) node[color=red]{$(1,2)'$};
\draw (7.5,1) node[color=red]{$(-2,1)'$};
\draw (9,2.5) node[color=red]{$(-1,2)'$};
\draw (7.5,-1) node[color=red]{$(-2,-1)'$};
\draw (9,-2.5) node[color=red]{$(-1,-2)'$};
\draw (12.5,-1) node[color=red]{$(2,-1)'$};
\draw (11,-2.5) node[color=red]{$(1,-2)'$};

 \end{tikzpicture}
 
\caption{\label{fig:pattern_SLOPE} Pattern equivalence classes for
SLOPE in $p=2$ dimensions: On the left, the blue polytope is $B^* =
\dslope(0) = \conv\{\pm(w_1,w_2)',\pm(w_1,-w_2)',
\pm(w_2,w_1)',\pm(w_2,-w_1)'\}$, the signed permutahedron for the
SLOPE weights $w_1 > w_2 > 0$, together with the (shifted, uncentered)
normal cones of the faces of $B^*$ in pink and green. The picture on
the right provides the actual (relative interior of the) normal cones
which are always centered at the origin and which, by
Theorem~\ref{thm:patt-cones}, coincide with the pattern equivalence
classes $C_\beta = \{b \in \R^p: \pattSLOPE(b) = \pattSLOPE(\beta)\}$
for the patterns $\{(0,0)', \pm(1,0)', \pm(0,1)', \pm(1, 1)', \pm (1,
-1)', \pm(1,2)', \pm (1,-2)', \pm (2,1)', \pm (2,-1)'\}$.}
\end{figure}
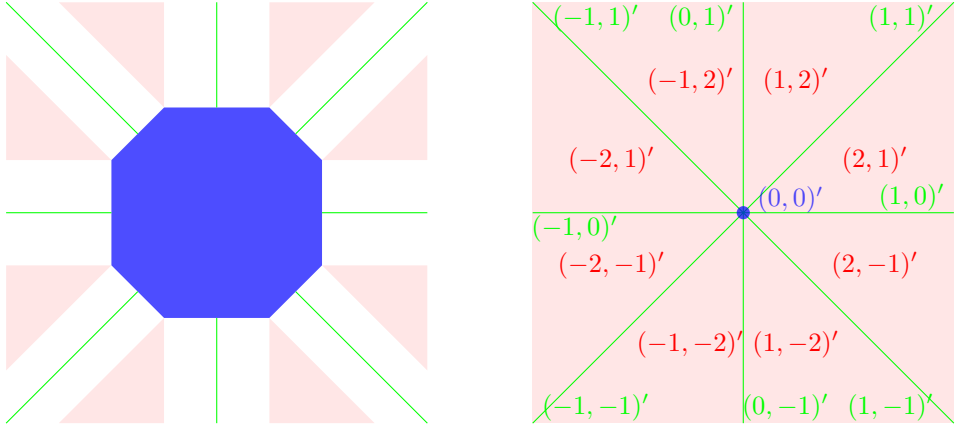


\item[{\bf $\ell_\infty$-norm:}] The subdifferential of the
$\ell_\infty$-norm at $0$ is the unit ball of the $\ell_1$-norm, $B^*
= \partial_{\|.\|_\infty}(0) = \{s : \|s\|_1 \leq 1\}$. The pattern of
$\beta \in \R^p$ can be represented by $\patt_\infty(\beta) \in
\{-1,0,1\}^p$ where each component is defined as
$$
\patt_\infty(\beta)_j = \ind\{\beta_j = \|\beta\|_\infty\} - \ind\{\beta_j = -\|\beta\|_\infty\}.
$$
Note that a zero component of $\patt_\infty(\beta)$ represents a
component of $\beta$ that is not maximal in absolute value or a
component of the zero vector. For instance, for $\beta =
(1.45,1.45,0.56,0,-1.45)'$, the pattern is given by
$\patt_\infty(\beta) = (1,1,0,0,-1)'$. Indeed, the subdifferentials at
two points in $\beta, \tilde\beta \in \R^p$ will be the same if and
only if $\patt_\infty(\beta) = \patt_\infty(\tilde\beta)$ so that
$C_\beta = \{b \in \R^p : \patt_\infty(b) = \patt_\infty(\beta)\}$.
This shows that the sup-norm patterns carry information about maximal
(in absolute value) and non-maximal components, as well as the sign
information of the maximal coefficients. The complexity of
$\patt_\infty(\beta)$ is given by $\ind\{\beta \neq 0\} (\sum_{j=1}^p
\ind\{|\beta_j| < \|\beta\|_\infty\} + 1)$, the number of non-maximal
components plus $1$ (accounting for the cluster of maximal components)
in case $\beta \neq 0$ and $0$ otherwise\footnote{An explicit
expression for $\dsup(\beta)$ can be found in
Appendix~\ref{subapp:nrc-sup}.}. See also Figure~\ref{fig:pattern_supnorm}.

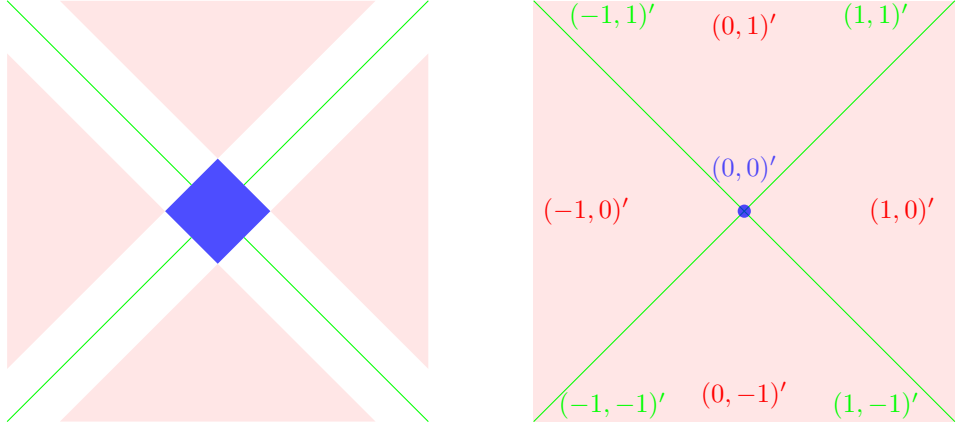
\begin{figure}[h!]
   
\centering

\begin{tikzpicture}[scale=0.7]

\fill[color=blue,fill opacity=0.7] (1,0) -- (0,1) -- (-1,0) -- (0,-1)--(1,0);
\fill[color=red,fill opacity=0.1] (0,1) -- (-3,4) -- (3,4) --(0,1);
\fill[color=red,fill opacity=0.1] (0,-1) -- (-3,-4) -- (3,-4) --(0,-1);
\fill[color=red,fill opacity=0.1] (1,0) -- (4,-3) -- (4,3) --(1,0);
\fill[color=red,fill opacity=0.1] (-1,0) -- (-4,-3) -- (-4,3) --(-1,0);

\draw[color=green] (0.5,0.5)-- (4,4);
\draw[color=green] (-0.5,-0.5)-- (-4,-4);
\draw[color=green] (0.5,-0.5)-- (4,-4);
\draw[color=green] (-0.5,0.5)-- (-4,4);

\fill[color=red,fill opacity=0.1] (6,-4) -- (6,4) -- (14,4) -- (14,-4)--(6,-4);

\draw[color=green] (6,-4)-- (14,4);
\draw[color=green] (6,4)-- (14,-4);
\draw (10,0) node[color=blue,fill opacity=0.7,circle,fill,inner sep=0pt,minimum size=5pt]{};
\draw (10,3.5) node[color=red]{$(0,1)'$};
\draw (10,-3.5) node[color=red]{$(0,-1)'$};
\draw (7,0) node[color=red]{$(-1,0)'$};
\draw (13,0) node[color=red]{$(1,0)'$};
\draw (12.5,3.7) node[color=green]{$(1,1)'$};
\draw (12.5,-3.7) node[color=green]{$(1,-1)'$};
\draw (7.5,3.7) node[color=green]{$(-1,1)'$};
\draw (7.5,-3.7) node[color=green]{$(-1,-1)'$};
\draw (10,0.8) node[color=blue,fill opacity=0.7]{$(0,0)'$};

\end{tikzpicture}
 
\caption{\label{fig:pattern_supnorm} Pattern equivalence classes for
the sup-norm in $p=2$ dimensions: On the left, the blue polytope is
$B^* = \dsup(0) = \conv\{\pm(1,0)',\pm(0,1)'\}$, together with the
(shifted, uncentered) normal cones of the faces of $B^*$ in pink and
green. The picture on the right provides the actual (relative interior
of the) normal cones which are always centered at the origin and
which, by Theorem~\ref{thm:patt-cones}, coincide with the pattern
equivalence classes $C_\beta = \{b \in \R^p: \patt_\infty(b) =
\patt_\infty(\beta)\}$ for the patterns $\patt_\infty(\beta) \in
\{(0,0)',\pm(0,1)',\pm(1, 0)',\pm(1,1)',\pm(1,-1)'\}$.}

\end{figure}


\item[{\bf Generalized LASSO:}] For the generalized Lasso, the penalty
term is given by $\pen(b) = \|Db\|_1$ where $D \in \R^{m \times p}$.
Note that, when $\ker(D) \neq \{0\}$, $\pen$ is only a semi-norm. We
list two common choices of $D$. For the subdifferential at $0$, we
have $\dgen(0)=D'[-1,1]^m$, see \citet[p.184]{HiriartLemarechal01}.

\begin{enumerate}

\item Let $p \geq 2$ and let $\Dtv \in \R^{(p-1) \times p}$ be the
first-order difference matrix defined as 
$$
\Dtv = \begin{pmatrix}
-1 & 1 & 0 & \dots & 0\\
0 & -1 & 1 & \dots & 0 \\
\vdots & \ddots &\ddots & \ddots & \vdots \\
0 & \dots & 0 & -1 & 1
\end{pmatrix}.
$$
The subdifferentials $\dgtv(\beta)$ and $\dgtv(\tilde\beta)$ are equal
if and only if $\sign(\Dtv\beta) = \sign(\Dtv\tilde\beta)$, so that we
can represent the pattern by this expression. Note that
$\sign(\Dtv\beta)_j = 0$ if $\beta_{j+1} = \beta_j$. Moreover,
$\sign(\Dtv\beta)_j = 1$ or $\sign(\Dtv\beta)_j = -1$ if $\beta_{j+1}
> \beta_j$ or $\beta_{j+1} < \beta_j$, respectively. For example, the
pattern of $\beta = (1.45,1.45,0.56,0.56,-0.45,0.35)'$ is given by
$\pattTV(\beta) = \sign(\Dtv\beta) = (0,-1,0,-1,1)'$. Clearly,
$C_\beta = \{b \in \R^p : \pattTV(b) = \pattTV(\beta)\}$. The
complexity of $\pattTV(\beta)$ is given by $1 +
\|\sign(\Dtv\beta)\|_1$, the number of jumps plus 1 (accounting for
the last component). See also Figure~\ref{fig:pattern_TV}.

\begin{figure}[h!]

\centering

\begin{tikzpicture}[scale=0.7]

\draw[color=blue,fill opacity=0.7] (1,-1)-- (-1,1);

\fill[color=red,fill opacity=0.1] (2,4) -- (-4,-2) -- (-4,4) -- (2,4);
\fill[color=red,fill opacity=0.1] (-2,-4) -- (4,2) -- (4,-4) -- (-2,-4);

\draw[color=green] (-4,-4)-- (4,4);

\fill[color=red,fill opacity=0.1] (6,-4) -- (6,4) -- (14,4) -- (14,-4) -- (6,-4);

\draw[color=green] (6,-4)-- (14,4);
\draw (13.5,3) node[color=green]{$0$};
\draw (6.5,3.5) node[color=red]{$1$};
\draw (13.5,-3.5) node[color=red]{$-1$};

\end{tikzpicture}
 
\caption{\label{fig:pattern_TV} Pattern equivalence classes for the
generalized LASSO with penalizing first-order differences ($D = \Dtv$)
in $p=2$ dimensions: On the left, the blue polytope is
$B^*=\dpen(0)=\conv\{\pm(1,-1)'\}$ together with the (shifted,
uncentered) normal cones of the faces of $B^*$ in pink and green. The
picture on the right provides the actual (relative interior of the)
normal cones which are always centered at the origin and which, by
Theorem~\ref{thm:patt-cones}, coincide with the pattern equivalence
classes $C_\beta = \{b \in \R^p: \pattTV(b) = \pattTV(\beta)\}$ for
the patterns $\pattTV(\beta) \in \{-1,0,1\}$.}

\end{figure}
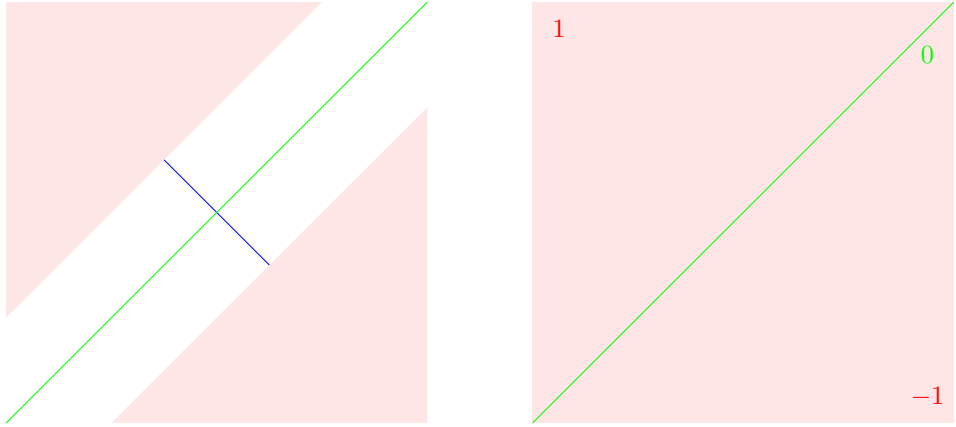

\item Let $p \geq 3$ and let $\Dtf \in \R^{(p-2) \times p}$
be the second-order difference matrix defined as
$$
\Dtf = \begin{pmatrix}
1 & -2 & 1 & 0 & \dots & 0\\
0 & 1 & -2 & 1 & \dots & 0 \\
\vdots & \ddots &\ddots & \ddots & \ddots & \vdots \\
0 & \dots & 0 & 1 & -2 & 1
\end{pmatrix}.
$$
The resulting method is called $\ell_1$-trend filtering
\citep{KimEtAl09} which in this context can be viewed as a special
case of the generalized LASSO. The subdifferentials $\dgtf(\beta)$ and
$\dgtf(\tilde\beta)$ are equal if and only if $\sign(\Dtf\beta) =
\sign(\Dtf\tilde\beta)$, so that we can represent the pattern by this
expression. To illustrate this pattern structure, consider the
piecewise linear curve $G_\beta =
\cup_{j=1}^{p-1}[(j,\beta_j),(j+1,\beta_{j+1})]$. Note that
$\sign(\Dtf\beta)_j = 0$ if, in a neighborhood of the point
$(j,\beta_j)$, the curve $G_\beta$ is linear. Moreover,
$\sign(\Dtf\beta)_j = 1$ or $\sign(\Dtf\beta)_j = -1$ if, in a
neighborhood of the point $(j,\beta_j)$, the curve $G_\beta$ convex or
concave, respectively. For instance, Figure~\ref{fig:pattern_TF}
provides an illustration of $\sign(D^{\rm tf}(x))$ for a particular
$x\in \R^9$. Finally, the complexity of $\pattTF(\beta)$ is given by
is given by $2+\|\sign(\Dtf\beta)\|_1$, the number ``non-linear
points'' plus $2$ (accounting for the first and last point). See also
Figure~\ref{fig:pattern_TF}.

\end{enumerate}

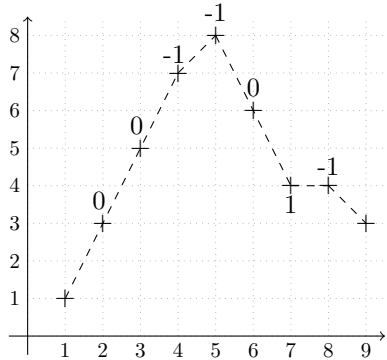
\begin{figure}[h!]

\centering

\begin{tikzpicture}[scale=0.5]

\draw[very thin,color=gray,dotted] (0,0) grid (9.5,8.5);
\draw[->] (-0.5,0) -- (9.5,0);
\draw[->] (0,-0.5) -- (0,8.5);
\foreach \y in {1,2,...,8}
\draw (0,\y) node[left,scale=0.8]{\y};
\foreach \x in {1,2,...,9}
\draw (\x,0) node [below,scale=0.8] {\x};

\draw (1,1) node {$+$};
\draw (2,3) node {$+$};
\draw (1.9,3.6) node {0};
\draw (3,5) node {$+$};
\draw (2.9,5.6) node {0};
\draw (4,7) node {$+$};
\draw (3.9,7.5) node {-1};
\draw (5,8) node {$+$};
\draw (5,8.6) node {-1};
\draw (6,6) node {$+$};
\draw (6,6.6) node {0}; 
\draw (7,4) node {$+$};
\draw (7,3.5) node {1};
\draw (8,4) node {$+$};
\draw (8,4.5) node {-1};
\draw (9,3) node {$+$};

\draw[dashed] (1,1) -- (4,7);
\draw[dashed] (4,7) -- (5,8);
\draw[dashed] (5,8) -- (7,4);
\draw[dashed] (7,4) -- (8,4);
\draw[dashed] (8,4) -- (9,3);

\end{tikzpicture}  

\caption{\label{fig:pattern_TF} In this figure, the dotted curve
represents $G_\beta$ described above for $\beta = (1,3,5,7,8,6,4,4,3)'$. Here,
$\sign(\Dtf\beta) = (0,0,-1,-1,0,1,-1)'$.}

\end{figure} 

\end{description}

Note that for the $\ell_1$-norm, the sorted-$\ell_1$-norm and the
$\ell_\infty$-norm, the pattern of $\beta$ itself is a canonical
representative of the equivalence class $C_\beta$. On the other hand,
for generalized LASSO, $\sign(\Dtv\beta)$ and $\sign(\Dtf\beta)$
characterize the pattern but are not an element of $C_\beta$ as there
appears to be no natural way to represent the pattern as such.

\end{example}

\section{Pattern Recovery in Penalized Estimation: Accessibility,
Noiseless Recovery and the Irrepresentability Condition}
\label{sec:recovery}

We now turn to the question of under which conditions a pattern can be
recovered by a penalized estimation procedure. For this, we first
introduce the notion of accessible patterns. Accessibility requires
the existence of \emph{a} response vector such that the resulting
estimator exhibits the required pattern, which is clearly a minimal
condition for possible pattern recovery. The definition generalizes
the notion of accessible sign vectors for LASSO
\citep{SepehriHarris17,SchneiderTardivel22} and accessible patterns
for SLOPE \citep{SchneiderTardivel22} to the general class of
estimators penalized by a polyhedral gauge.

\begin{definition}[Accessible pattern] 
Let $X \in \R^{n \times p}$, $\lambda > 0$ and $\pen$ be a real-valued
polyhedral gauge. We say that $\beta \in \R^p$ has an accessible
pattern with respect to $X$ and $\lambda\pen$, if there exists $y \in
\R^n$ and $\hat\beta \in \SpenL(y)$ such that $\hat \beta
\overset{\pen}{\sim}\beta$. 
\end{definition}
Proposition~\ref{prop:acc} provides both a geometric and an analytic
characterization for the notion of accessible patterns.

\begin{proposition}[Characterization of accessible patterns] 
\label{prop:acc} Let $X \in \R^{n \times p}$, $\lambda>0$ and
$\pen:\R^p \to \R$ be a polyhedral gauge.

\begin{enumerate}

\item Geometric characterization: The pattern of  $\beta \in
\R^p$ is accessible with respect to $X$ and $\lambda\pen$ if and only
if
$$
\row(X) \cap \dpen(\beta) \neq \emptyset.
$$

\item Analytic characterization: The pattern of $\beta \in \R^p$
is accessible with respect to $X$ and $\lambda\pen$ if and only if for
any $b \in \R^p$ the implication
$$
X\beta = Xb \implies \pen(\beta) \leq \pen(b)
$$
holds.

\end{enumerate}

\end{proposition}

Based on Proposition~\ref{prop:acc}, it is clear that accessibility
does not depend on the particular value of the tuning parameter
$\lambda$. We therefore also say that the pattern of $\beta$ is
accessible with respect to $X$ and $\pen$. The geometric
characterization shows that we have accessibility for the pattern of
$\beta$ if and only if $\row(X)$ intersects the face of $B^*$ that
corresponds to the pattern of $\beta$. Since the smaller the
complexity of the pattern is, the larger the dimension of the
corresponding face becomes, the geometric criterion also demonstrates
that more complex patterns are harder to access. 

The following proposition and corollary greatly strengthen the notion
of accessibility, showing that under uniform uniqueness, accessibility
already implies the existence an entire set of $y's$ in $\R^n$ with
non-empty interior that lead to the pattern of interest:

\begin{proposition} \label{prop:attainability}
Let $X \in \R^{n \times p}$, $\lambda > 0$ and $\pen:\R^p \to \R$ be a
polyhedral gauge. Assume that uniform uniqueness holds, i.e. for any
$y \in \R^n$, the set $\SpenL(y)$ contains the unique minimizer
$\hat\beta(y)$. Let $\beta \in \R^p$. If the pattern of $\beta$ is
accessible with respect to $X$ and $\pen$, the set
$$
A_\beta = \{y : \hat\beta(y) \penequiv \beta\} 
$$
has non-empty interior.
\end{proposition}

Clearly, Proposition~\ref{prop:attainability} demonstrates that under
a uniqueness assumption, accessibility of a pattern already implies
that the pattern can be detected by the penalized procedure with
positive probability, provided that $y$ is generated by a continuous
distribution taking on all values in $\R^n$. This is summarized in the
corollary below. It is related to the concept of attainability in
\cite{HejnyEtAl23TR} which they view for SLOPE in an asymptotic
setting.

\begin{corollary} \label{cor:attainability}
Let $X \in \R^{n \times p}$, $\lambda > 0$ and $\pen:\R^p \to \R$ be a
polyhedral gauge. Let $\beta \in \R^p$ have an accessible pattern and
assume that uniform uniqueness holds. If $y$ follows a distribution
with positive Lebesgue-density on $\R^n$, then
$$
\mathbb{P}(\hat\beta(y) \penequiv \beta) > 0. 
$$
\end{corollary}

We now turn to a stronger requirement for pattern recovery. For this,
we consider the solution path of a penalized estimator, given by the
curve $0 < \lambda \mapsto \hat\beta_\lambda$, where
$\hat\beta_\lambda$ is the (assumed to be unique) element of
$\SpenL(y)$ for fixed $y \in \R^n$ and $X \in \R^{n \times p}$.
Definition~\ref{def:nrc} below alludes to the notion of a solution
path. Note, however, that Definition~\ref{def:nrc} does not require
uniqueness of estimator.

\begin{definition}[Noiseless recovery condition] \label{def:nrc}
Let $\pen$ be a real-valued polyhedral gauge, $X \in \R^{n\times p}$
and $\beta \in \R^p$. We say that the pattern of $\beta$ satisfies the
noiseless recovery condition with respect to $X$ and $\pen$ if
$$
\exists \lambda > 0, \exists \hat\beta \in \SpenL(X\beta) \text{ \rm
such that } \hat \beta \overset{\pen}{\sim} \beta.
$$
\end{definition}
For instance, $\beta = 0$ satisfies the noiseless recovery condition
with respect to $X$ and $\pen$ since then $X\beta = 0$ and $0 \in
\SpenL(0)$. Another way of stating the noiseless recovery condition is
to require that in the noiseless case $Y = X\beta$, the solution path
contains a minimizer having the same pattern as $\beta$. The noiseless
recovery condition is illustrated for the supremum norm in
Figure~\ref{fig:path-sup} for the particular case where $X$ and
$\beta$ are given by
$$
X = \begin{pmatrix} 1 & 0 & 2 \\ 0 & 1 & 1 \end{pmatrix}
\text{ \rm and } \beta=(0,2,2)'.
$$

\begin{figure}[h!]
\centering
\begin{tikzpicture}[scale=0.5]
\draw (8.5,11) node[color=black]{\Large{Solution path for the supremum norm}};
\draw[very thin,color=gray,dotted] (-0.5,-0.5) grid (18,11.5);
\draw[->] (-0.5,0) -- (18,0) node[right] {$\lambda$};
\draw[->] (0,-0.5) -- (0,10.5);
\draw[color=red,dotted,line width=2pt] (0,10) -- (2,6.666);
\draw[color=blue,dotted,line width=2.5pt] (0,10) -- (2,6.666);
\draw[color=black,dotted,line width=1.5pt] (0,0) -- (2,6.666);
\draw[color=red,dotted,line width=2pt] (2,6.666) -- (15,0);
\draw[color=blue,dotted,line width=2.5pt] (2,6.666) -- (15,0);
\draw[color=black,dotted,line width=1.5pt] (2,6.666) -- (15,0);
\draw[color=red,dotted,line width=2pt] (15,0) -- (17.5,0);
\draw[color=blue,dotted,line width=2.5pt] (15,0) -- (17.5,0);
\draw[color=black,dotted,line width=1.5pt] (15,0) -- (17.5,0);
\draw[dashed] (15,0) -- (15,10.5);
\draw[dashed] (2,0) -- (2,10.5);
\draw (15,0) node[below]{$20$};
\draw (2,0) node[below]{$8/3$};
\draw (0,5) node[left]{$1$};
\draw (0,10) node[left]{$2$};

\draw[color=black,dotted,line width=1.5pt] (2.5,10) -- (3,10);
\draw (5.5,10) node[color=black]{1st component};
\draw[color=red,dotted,line width=2pt] (2.5,9.2) -- (3,9.2);
\draw (5.5,9.2) node[color=red]{2nd component};
\draw[color=blue,dotted,line width=2.5pt] (2.5,8.4) -- (3,8.4);
\draw (5.5,8.4) node[color=blue]{3rd component};
\draw (16.5,1.5) node{\small{${\rm patt}_{\infty}=$}};
\draw (16.5,0.5) node{\small{$(0,0,0)'$}};
\draw (7.5,1.5) node{\small{${\rm patt}_{\infty}=$}};
\draw (7.5,0.5) node{\small{$(1,1,1)'$}};
\draw (1.2,1.5) node{\small{${\rm patt}_{\infty}=$}};
\draw (1,0.5) node{\small{$(0,1,1)'$}};
 \end{tikzpicture}
 
\caption{\label{fig:path-sup} Shown are the curves of the three
component functions $\lambda \mapsto \hat\beta_{\lambda,1}$ (black
dotted curve), $\lambda \mapsto \hat\beta_{\lambda,2}$ (red dotted
curve) and $\lambda \mapsto \hat\beta_{\lambda,3}$ (blue dotted curve)
for $\lambda > 0$, where $\{\hat\beta_\lambda\} = \SsupL(X\beta)$.
Note that $\patt_\infty(\beta)$ satisfies the noiseless recovery
condition. Indeed, $\patt_\infty(\hat\beta_\lambda) = (0,1,1)'$ for
$\lambda \in (0,8/3)$.}

\end{figure}
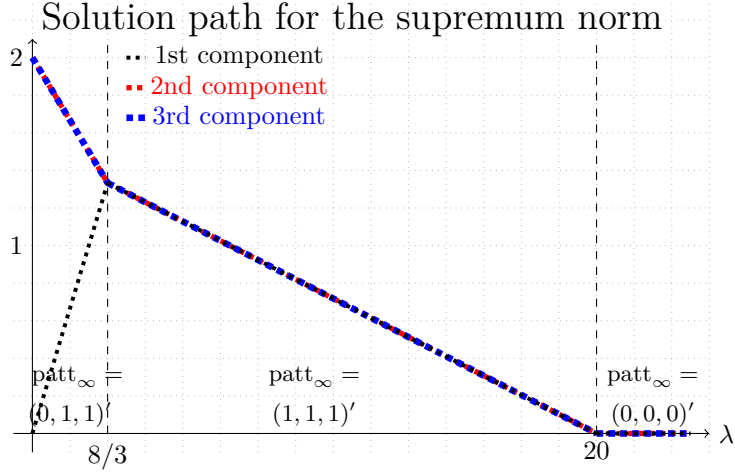 

In Proposition~\ref{prop:nrc-geom} in Section~\ref{sec:recovery}, we
prove that the noiseless recovery condition occurs if and only if
$X'X\lin(C_\beta) \cap \dpen(\beta) \neq \emptyset$. Based on this
characterization, it is clear that the condition indeed depends on $\beta$
only through its pattern. For an analytic expression for checking the
noiseless recovery condition, some formulas are given in the
literature. For example, when $\pen = \|.\|_1$, the noiseless recovery
condition can be shown to be equivalent to
\begin{equation} \label{eq:irrep_lasso}
\|X'(X_I')^+\sign(\beta_I)\|_\infty \le 1 \text{ \rm and }\sign(\beta_I)\in \row(X_I),    
\end{equation}
where $I =\{j \in [p] : \beta_j \neq 0\}$. Note that if $\ker(X_I) =
\{0\}$, we have $\sign(\beta_I)\in \row(X_I)$ and
expression~\eqref{eq:irrep_lasso} coincides with the well-known
\emph{irrepresentability} or \emph{mutual incoherence condition or non-degenerate condition} for
the LASSO given by $\|X_{I^c}'X_I(X_I'X_I)^{-1}\sign(\beta_I)\|_\infty
\leq 1$ \citep{BuehlmannVdGeer11,Wainwright09,Zou06,ZhaoYu06}. Thus,
the irrepresentability condition for the LASSO can be thought of as an
analytical shortcut for checking the noiseless recovery condition. For
the sorted-$\ell_1$-norm, when $m = \pattSLOPE(\beta)$, the noiseless
recovery condition is equivalent to
$$
\|X'(\tilde X'_m)^+\tilde W_m\|^*_w \le 1 \text{ \rm and } \tilde W_m
\in \row(\tilde X_m),
$$
where $\|.\|^*_w$ is the dual sorted-$\ell_1$-norm,  $\tilde X_m$ is
the so-called clustered matrix and $\tilde W_m$ is the clustered
parameter, see \cite{BogdanEtAl25} for details or
\cite{VaiterEtAl15,VaiterEtAl18} for similar expressions. In
Proposition~\ref{prop:nrc-sup} in Appendix~\ref{subapp:nrc-sup}, we
also provide an analytic characterization of the noiseless recovery
condition for the supremum norm: Let $I =\{j \in [p]: |\beta_j| <
\|\beta\|_\infty\}$, $\tilde X = (\tilde X_1|X_{ I})$ where $\tilde
X_1 = X_{I^c}\sign(\beta_{I^c})$. The noiseless recovery condition
holds if and only if
$$
\|X'(\tilde X')^+e_1\|_1 \leq 1 \text{ \rm and } e_1 \in \row(\tilde X), 
\text{ \rm where } e_1 = (1,0,\dots,0)'.
$$
Figure~\ref{fig:path-sup} confirms this characterization. Indeed, in
the above example we have
$$
\tilde X = \begin{pmatrix} 2 & 1 \\ 2 & 0  \end{pmatrix}, \; e_1=(1,0)' \text{ \rm and }
X'(\tilde X')^+ e_1 = (0,1/2,1/2)'
$$
and based on Figure~\ref{fig:path-sup} one may observe that the
noiseless recovery condition holds for $\beta$. Subsequently, we show
that

\begin{enumerate}

\item The noiseless recovery condition is a necessary condition for
pattern recovery with probability larger than $1/2$, see
Theorem~\ref{thm:irrep-cond}.

\item Thresholded penalized estimators recover the pattern of $\beta$
under much weaker condition than the noiseless recovery condition, see
Section~\ref{sec:recovery-thres}.

\end{enumerate}

\begin{theorem} \label{thm:irrep-cond}
Let $Y = X\beta + \eps$ where $X \in \R^{n \times p}$ is a fixed
matrix, $\beta \in \R^p $ and $\eps$ follows a symmetric distribution.
Let $\pen$ be a real-valued polyhedral gauge. If $\beta$ does not
satisfy the noiseless recovery condition with respect to $X$ and
$\pen$, then
$$
\P\left(\exists \lambda > 0 \; \exists \hat\beta \in
S_{X,\lambda\pen}(Y) \text{ \rm such that } \hat\beta \penequiv
\beta\right) \leq 1/2.
$$
\end{theorem}

By Theorem~\ref{thm:irrep-cond}, if the noiseless recovery condition
does not hold for the LASSO (for example, when $\|X_{\overline
I}'X_I(X_I'X_I)^{-1}\sign(\beta_I)\|_\infty > 1$), then
$$
\P(\exists \lambda>0 \; \exists \hat\beta \in \SlassoL(Y) \text{ \rm
such that } \sign(\hat\beta) = \sign(\beta)) \leq 1/2.
$$
This above result is stronger than the one given in Theorem~2 in
\cite{Wainwright09} which shows that $\P(\sign(\betaLASSO(\lambda)) =
\sign(\beta)) \leq 1/2$ for fixed $\lambda > 0$.
Theorem~\ref{thm:irrep-cond} demonstrates that the noiseless recovery
condition can be viewed as a unified irrepresentability condition in a
general penalized estimation framework which may also be seen as more
interpretable than typical representations of the irrepresentability
condition. 

\subsection*{The gap between accessibility and noiseless recovery}

Clearly, if the pattern of $\beta$ satisfies the noiseless recovery
condition with respect to $X$ and $\pen$, the pattern of $\beta$ is
accessible with respect to $X$ and $\pen$. Indeed, if noiseless
recovery occurs, there exists $\lambda_0 > 0$ and $\hat\beta \in
S_{X,\lambda_0\pen}(X\beta)$ such that $\hat\beta \penequiv \beta$,
consequently the pattern of $\beta$ is accessible with respect to $X$
and $\lambda_0\pen$ or, equivalently, with respect to $X$ and $\pen$.
The following proposition gives a geometric criterion for noiseless
recovery and allows to better understand the connection between
accessibility and noiseless recovery.

\begin{proposition} \label{prop:nrc-geom} 
Let $X \in \R^{n \times p}$ and $\pen:\R^p \to \R$ be a polyhedral
gauge. Let $\beta \in \R^p$. Then $\beta$ satisfies the noiseless
recovery condition with respect to $X$ and $\pen$ if and only if
$$
X'X\lin(C_\beta) \cap \dpen(\beta) \neq \emptyset.
$$
\end{proposition}

For more insights into the connection to accessibility, note that by
Proposition~\ref{prop:acc}, accessibility holds if and only if
$\row(X) \cap \dpen(\beta) \neq \emptyset$ and, loosely speaking,
accessibility is harder to satisfy if the complexity of the pattern is
large since then the dimension of $\dpen(\beta)$ is small. However,
since $\row(X) = \col(X') = \col(X'X)$, the accessibility criterion
can be understood as $X'X\R^p \cap \dpen(\beta) \neq 0$, showing that
when the dimension of $\lin(C_\beta)$, i.e. the complexity of the
pattern of $\beta$, is large, the gap between accessibility and
noiseless recovery becomes small. This is illustrated in
Figure~\ref{fig:slope_acc-irr} for a simple SLOPE pattern in
Section~\ref{sec:num-exp} yielding a large gap between accessibility
and noiseless recovery.

In the following section, we show that thresholded penalized
least-squares estimators recover the pattern of $\beta$ under the
accessibility condition only, provided that the noise is small enough
and uniqueness holds.

\section{Nested Patterns and Pattern Recovery by Thresholded Penalized
Estimation} \label{sec:recovery-thres}

In practical applications, such as in genetics, many columns of the
design matrix $X$ may be irrelevant, implying that most components of
$\beta$ could be null. It is well known that the $\ell_1$-norm is the
appropriate convex penalty to promote sparsity
\citep{TraonmilinEtAl24}, making the LASSO the natural estimator for
such applications. However, even with a well chosen tuning parameter,
the LASSO estimator might still not have enough zero components to
sufficiently discard irrelevant columns. It can therefore be natural
to set small components of $\betaLASSO$ to zero and so consider the
thresholded LASSO estimator $\betaLASSOtau$ for some threshold $\tau
\geq 0$. In fact, if the threshold is appropriately selected, the
estimator allows to recover $\sign(\beta)$, the LASSO pattern of
$\beta$, under weaker conditions than LASSO itself
\citep{TardivelBogdan22}. This is because the LASSO tends to
``overfit'' the sign of $\beta$ for a sufficiently strong signal of
$\beta$. In this section, we give the mathematical explanation of this
phenomenon within the general penalized estimation framework and
discuss what conclusions may be drawn for a general concept of
``thresholding'' for a penalized estimator.

Along these lines, note that for any threshold $\tau \geq 0$, the
inclusion $\dlasso(\betaLASSO) \subseteq \dlasso(\betaLASSOtau)$
holds. This observation is helpful for understanding the following
theorem and to motivate how to ``threshold'' in a general setting.

\begin{theorem} \label{thm:overfit}
Let $\pen$ be a real-valued polyhedral gauge, $X \in \R^{n \times p}$
and $\beta \in \R^p$. Let $y^{(r)} = X\beta + \eps^{(r)}$ where
$(\eps^{(r)})_{r\in \N}$ is a sequence in $\R^n$ such that $\lim_{r\to
\infty} \eps^{(r)}=0$. Assume that uniform uniqueness holds and let
$\hat\beta^{(r)}$ be the unique minimizer in $\SpenLr(y^{(r)})$, where
$\lim_{r \to \infty} \lambda^{(r)} = 0$. If the pattern of $\beta$ is
accessible with respect to $X$ and $\pen$, there exists $r_0 \in \N$
such that for all $r \geq r_0$
$$
\dpen(\widehat \beta^{(r)}) \subseteq \dpen(\beta).
$$
\end{theorem}
Theorem~\ref{thm:overfit} can be seen to corroborate Theorem~1 in
\cite{FadiliEtAl19} using mirror-stratifiable regularizers for the
particular case where the stratification is a partition of pattern
classes. However, the above Theorem~\ref{thm:overfit} and Theorem~1 in
\cite{FadiliEtAl19} significantly differ in both the asymptotic
regimes they consider as well as in the assumptions needed: In
Theorem~1 of \cite{FadiliEtAl19}, the number of observations, $n$,
tends to $+\infty$, whereas in Theorem~\ref{thm:overfit}, the design
matrix $X \in \R^{n \times p}$ is fixed and the noise tends to $0$. In
terms of assumptions on $\beta$, for Theorem~1 of \cite{FadiliEtAl19},
for the case when $y$ is the response of a linear regression model,
$\beta$ is the unique minimizer of a positive semi-definite quadratic
form for which $\pen$ is minimal, whereas the pattern of $\beta$ is
required to be accessible in Theorem~\ref{thm:overfit}.

For interpretation of the above theorem, we define a \emph{nesting
structure} for patterns by saying that the pattern of $\beta$ is
\emph{nested} in the pattern of $\tilde\beta$ if $\dpen(\tilde\beta)
\subseteq \dpen(\beta)$. Notice that this implies that the pattern of
$\beta$ is less complex than the pattern of $\tilde\beta$. In this
sense, Theorem~\ref{thm:overfit} shows that given uniqueness and
accessibility, the correct pattern is nested in the pattern of the
penalized estimator, provided that the signal of $\beta$ is large
enough. For the LASSO, this complies with Theorem~1 in
\cite{FadiliEtAl19} or Theorem~2 in \cite{PokarowskiEtAl22} proving
that, asymptotically, the support of the LASSO contains the support of
$\beta$. Another way of interpreting the theorem is that, for exact
pattern recovery, a penalized estimator should be ``thresholded'' like
the LASSO, by moving to a nested, less complex pattern. We illustrate
below what thresholding means for penalized estimators other than the
LASSO.

\begin{enumerate}

\item The penalty term $\|\cdot\|_{\infty}$ promotes clustering of
components that are maximal in absolute value: Once $|\hat\beta_j|<
\|\hat\beta\|_\infty$ but $|\hat\beta_j| \approx
\|\hat\beta\|_\infty$, it is quite natural to set $|\hat\beta_j| =
\|\hat\beta\|_\infty$. Let $\betathres$ be the estimator taking into
account this approximation, obtained after slightly modifying
$\hat\beta$. Then $\dsup(\hat\beta) \subseteq \dsup(\betathres)$.

\item The sorted-$\ell_1$-norm penalty promotes clustering of
components equal in absolute value: Once $|\betaSLOPE_j| \approx
|\betaSLOPE_i|$, it is quite natural to set $|\betaSLOPE_i|=
|\betaSLOPE_j|$. Let $\betathres$ be the estimator taking into account
this approximation and obtained after slightly modifying $\betaSLOPE$.
Then, $\dslope(\betaSLOPE) \subseteq \dslope(\betathres)$.

\item The penalty term $\|D^{tv}\cdot\|$ 
promotes neighboring components to be equal: Once $\hat\beta_j \approx
\hat \beta_{j+1}$, it is quite natural to set $\hat\beta_j =
\hat\beta_{j+1}$. Let $\betathres$ be the estimator taking into
account this approximation and obtained after slightly modifying
$\hat\beta$. Then, $\dgtv(\hat\beta) \subseteq \dgtv(\betathres)$.

\end{enumerate}

More precisely, we say that $\betathres$ is a thresholded version of
the estimator $\hat\beta$ with respect to the polyhedral gauge $\pen$
(or simply a thresholded estimator when there is no ambiguity), if the
pattern of $\betathres$ is nested in the pattern of $\hat\beta$,
namely when $\dpen(\hat\beta) \subseteq \dpen(\betathres)$. In view of
this general concept of thresholded estimators,
Theorem~\ref{thm:overfit} can also be read as showing how the
noiseless recovery condition from Section~\ref{sec:recovery} -- needed
for pattern recovery with probability larger than 1/2 -- can be
relaxed to the minimal condition of accessibility under a uniqueness
assumption when considering thresholded estimation: in fact, even sure
pattern recovery is then guaranteed, provided that the noise is small
enough.

For completeness, the proposition below shows that accessibility with
respect to a penalized estimator is equivalent to accessibility with
respect to a thresholded penalized estimator, demonstrating that the
accessibility in Theorem~\ref{thm:overfit} is indeed a necessary
condition.

\begin{proposition} \label{prop:acc-thres}
Let $\pen$ be a real-valued polyhedral gauge, $X \in \R^{n \times p}$,
$\lambda > 0$ and $\beta \in \R^p$. We have
\begin{align*}
\exists y \in \R^n,\ & \exists \hat\beta \in \SpenL(y) \text{ \rm
such that }  \hat\beta \penequiv \beta \\ 
\iff \; & \exists y \in \R^n, \exists \hat\beta \in \SpenL(y) \text{ \rm such
that } \dpen(\hat\beta) \subseteq \dpen(\beta).
\end{align*}
\end{proposition}

\section{A Necessary and Sufficient Condition for Uniform Uniqueness} \label{sec:unique}
In Proposition~\ref{prop:attainability},
Corollary~\ref{cor:attainability} and Theorem~\ref{thm:overfit} we
require uniform uniqueness, i.e., uniqueness of the penalized
optimization problem \eqref{eq:pen_problem} for a given $X \in \R^{n
\times p}$ for all $\lambda > 0$ and all $y \in \R^n$. We provide a
necessary and sufficient condition for this kind of uniqueness in
Theorem~\ref{thm:unique-gauge} below. This theorem relaxes the
coercivity condition for the penalty term needed in Theorem~1 in
\cite{SchneiderTardivel22} and extends the result to encompass methods
such as the generalized LASSO.

\begin{theorem}[Necessary and sufficient condition for uniform uniqueness] 
\label{thm:unique-gauge} 
Let $\pen$ be a real-valued polyhedral gauge, $X \in \R^{n \times p}$,
and $\lambda > 0$. Then the solution set $\SpenL(y)$ from
\eqref{eq:pen_problem} is a singleton for all $y \in \R^n$ if and only
if $\row(X)$ does not intersect a face of $B^*$ whose
dimension\footnote{The dimension of a face is defined as the dimension
of its affine hull, see Appendix~\ref{app:polytopes} for details.} is
strictly less than $\defect(X) = \dim(\ker(X))$.
\end{theorem}

\noindent
Note that a face $F$ of $B^*$ satisfies
$$
\dim(F) < \defect(X) \iff \codim(F) > \rk(X),
$$ 
where $\codim(F) = p - \dim(F)$. Using Corollary~\ref{cor:compl-codim}
and Proposition~\ref{prop:acc}, we may therefore conclude the
following result. The recent preprint of \cite{EverinkEtAl24TR}
generalizes Theorem~\ref{thm:unique-gauge} by proving that uniform
uniqueness implies continuity of the solution with respect to $y$.

\begin{corollary} \label{cor:unique-compl}
Let $\pen$ be a real-valued polyhedral gauge, $X \in \R^{n \times p}$,
and $\lambda > 0$. Then the optimization problem in
\eqref{eq:pen_problem} is uniquely solvable for all $y \in \R^n$ if
and only if no pattern with complexity exceeding $\rk(X)$ is
accessible.
\end{corollary}

For instance the generalized Lasso with total variation penalty is
uniquely solvable for all $y \in \R^n$ if and only if no pattern with
more than or equal to $\rk(X)$ jumps is accessible. For a better
understanding of whether uniqueness is a ``typical'' property, we give
the proposition stated below. Along these lines, observe that for the
methods listed in the introduction, the corresponding gauges are
symmetric: $\pen(b) = \pen(-b)$. For these cases,
Proposition~\ref{prop:unique-sym_gauge} shows that the set of
solutions of the penalized least squares problem is unbounded if
$\dim(\{b \in \R^p : \pen(b) = 0\}) > n$ or always a singleton if
$\dim(\{b\in \R^p : \pen(b)= 0\}) \leq n$.

\begin{proposition} \label{prop:unique-sym_gauge}
Let $\pen$ be a symmetric real-valued polyhedral gauge on $\R^p$,
namely, $\pen(b) = \max\{\pm u_1'b,\dots,\pm u_k'b\}$ for some
$u_1,\dots,u_k \in \R^p$. Then $\ker(\pen) = \{b \in \R^p : \pen(b) =
0\}$ is a vector space and the following holds.

\begin{enumerate}

\item \label{enum:X} If $\dim(\ker(\pen)) > n$, then
for any $X \in \R^{n \times p}$ the set $\SpenL(y)$ is unbounded for
all $y \in \R^n$ and all $\lambda > 0$.

\item \label{enum:mu} If $\dim(\ker(\pen)) \leq n$,
then
$$
\mu\left(\left\{X \in \R^{n\times p} : \exists y \in \R^n\; \exists \lambda > 0 \text{ with }
|\SpenL(y)| > 1\right\}\right) = 0,
$$
where $\mu$ is the Lebesgue measure on $\R^{n \times p}$.

\end{enumerate}

\end{proposition}

Proposition~\ref{prop:unique-sym_gauge} is consistent with the
findings in \cite{AliTibshirani19} for the special case of the
generalized Lasso.

We now illustrate a case of non-uniqueness occurring for the
generalized LASSO with $\pen(b) = \|Db\|_1$ for some $D \in \R^{m
\times p}$. Clearly,  $\ker(X) \cap \ker(D)= \{0\}$ is a necessary
condition for uniform uniqueness, yet, it is not sufficient, as
illustrated in the example below.

\begin{example}
An example of generalized LASSO optimization problem for which the set
of minimizers is not restricted to a singleton is given in
\cite{BarbaraEtAl19}:
$$
\Argmin_{b \in \R^p} \frac{1}{2}\|y-Xb\|_2^2+\frac{1}{2}\|Db\|_1 \text{ \rm where }
X = \begin{pmatrix} 1& 1 & 1 \\ 3 & 1 & 1 \\ \sqrt{2} & 0 & 0 \end{pmatrix},\;
D=\begin{pmatrix}1 & 1 & 0 \\ 1 & 0 & 1 \\ 2 & 1 & 1\end{pmatrix}
\text{ \rm and }y=\begin{pmatrix} 1 \\ 1 \\ 0 \end{pmatrix}.$$
Note that $S_{X,\frac{1}{2}\|D.\|_1}(y) =
\conv\{(0,1/2,0)',(0,0,1/2)'\}$. Since
$$
\|Db\|_1 = \max\{\pm(4b_1+2b_2+2b_3),\pm(2b_1+2b_2),\pm(2b_1+2b_3)\},
$$
we have $B^* = \conv\{\pm (4,2,2)',\pm (2,2,0)',\pm(2,0,2)'\}$.
Because the vertex $F = (4,2,2)'$ is an element of $\row(X)$ and
satisfies $\codim(F) = 3 - \dim(F) = 3 - 0 > 2 = \rk(X)$, uniform
uniqueness cannot hold. This of course complies with the fact that
$S_{X,\frac{1}{2}\|D.\|_1}(y)$ is not a singleton.
\end{example}
When $\ker(X) \cap \ker(D) = \{0\}$, in broad generality,  the set of
generalized LASSO minimizers is a polytope, i.e., a bounded polyhedron
\citep{BarbaraEtAl19}, and extremal points can be computed explicitly
\citep{DupuisVaiter23}. This description is relevant when the set of
minimizers is not a singleton.

\section{Numerical Illustrations} \label{sec:num-exp}

In this section, we present numerical experiments to illustrate the
connection between the accessibility and noiseless recovery condition
as well as the concept of thresholding based on the SLOPE method. The
simulations were carried out in Python. For our simulations, we fix
$\beta \in \{0,1\}^{784}$. The particular value we use is depicted in
Figure~\ref{fig:six_original}, giving a visualization when the vector
is reshaped as a plot of size $28 \times 28$, where $0$ represents a
white pixel and $1$ represents a black pixel.

\begin{figure}[h!]
\centering
\includegraphics[scale=0.6]{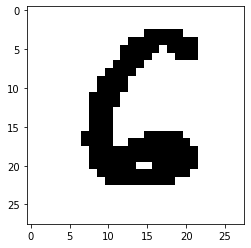}
\caption{The vector $\beta \in \{0,1\}^{784}$, reshaped as a picture of
size $28 \times 28$, represents the number six.}
\label{fig:six_original}

\end{figure}

For the weights of the sorted-$\ell_1$-norm, we choose $w =
(\sqrt{j}-\sqrt{j-1})_{1 \leq j \leq 784}$ as suggested in
\cite{Nomura20TR}.

\subsection*{Comparison between accessibility and noiseless recovery using SLOPE}

We first illustrate how likely it is that $\beta$ satisfies the
accessibility or the noiseless recovery condition. For this, we use
the following setup: Let $X \in \R^{n \times 784}$ be a matrix whose
coefficients are independent and identically $\mN(0,1/n)$-distributed.
Using Proposition~\ref{prop:acc}, we see that the probability that
$\beta$ is accessible with respect to $X$ and $\pen = \|.\|_w$ is
given by
$$
\P_X(\min\{\|b\|_w: Xb = X\beta\} = \|\beta\|_w).
$$
Furthermore, the probability that the noiseless recovery condition
holds for $\beta$ is
$$
\P_X(X'X\lin(C_\beta) \cap \dslope(\beta) \neq \emptyset)
$$
by Proposition~\ref{prop:nrc-geom}. Based on these formulae,
Figure~\ref{fig:slope_acc-irr} reports these probabilities as a
function of the number of rows of the matrix $X$. To approximate the
probabilities, we used $1000$ realizations of the matrix $X$. As can
be seen, accessibility is far more likely to occur than the noiseless
recovery condition.

\begin{figure}[h!]
\centering

\begin{subfigure}{0.49\textwidth}
\centering
\includegraphics[scale=0.4]{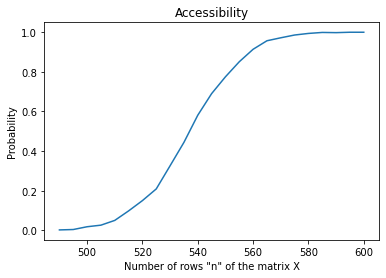}
\subcaption{\label{subfig:slope_accessibility}} 
\end{subfigure}
\begin{subfigure}{0.49\textwidth}
\centering
\includegraphics[scale=0.4]{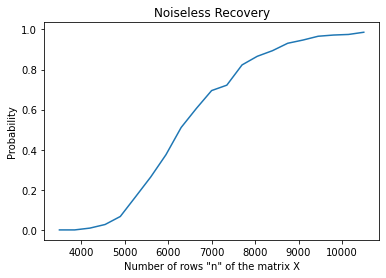}
\subcaption{\label{subfig:slope_irrepresentability}}
\end{subfigure}

\caption{\label{fig:slope_acc-irr} The probability of the
accessibility condition (\subref{subfig:slope_accessibility}) vs the
noiseless recovery condition
(\subref{subfig:slope_irrepresentability}) being satisfied for the
SLOPE pattern $\beta$ as a function of $n$, the number of rows of the
matrix $X$. While the shapes are qualitatively similar, note the
difference in ranges on the x-axes. The probability of accessibility
is almost zero when $n \leq 500$ and almost $1$ when $n \geq 600$. For
the noiseless recovery condition, the probability is almost zero when
$n \leq 4000$ and almost $1$ when $n \geq 10000$.}

\end{figure}

\subsection*{Pattern recovery by thresholded SLOPE}

For SLOPE, a practical way to construct  a thresholded estimator is to
apply the proximal operator of the sorted-$\ell_1$-norm to the SLOPE
estimator $\hat\beta$
$$
\prox_\tau(\hat\beta) = \Argmin_{b \in \R^p} \left\{\frac{1}{2}\|\hat\beta - b\|_2^2 + \tau\|b\|_w\right\},
$$
where $\tau \geq 0$ tunes the complexity of the thresholded estimator,
see \cite{TardivelEtAl20} and \cite{DupuisTardivel22} for an explicit
expression of this procedure. For the following, we consider the
Gaussian linear regression model $Y = X\beta + \eps$, where $X \in
\R^{600 \times 784}$ is a matrix whose coefficients are independent
and identically $\mN(0,1/600)$-distributed and the components of $\eps
\in \R^{600}$ are independent and identically distributed according to
a $\mN(0,0.05^2)$-distribution. For particular realizations $y$ of $Y$
and $X$, we denote $\hat\beta_\lambda$ the unique element of
$S_{X,\lambda\|.\|_w}(y)$ and $\lambda_\sure$ the tuning parameter
selected via the SURE formula for SLOPE \citep{Minami20}; namely
$\lambda_\sure$ is a minimizer of the function $\lambda > 0 \mapsto
\|y - X\hat\beta_\lambda\|_2^2 - 600\times 0.05^2 + 2 \times
0.05^2\|\pattSLOPE(\hat\beta_\lambda)\|_\infty$. In
Figure~\ref{fig:SLOPE_thres-SLOPE} we illustrate that, for $n=600$,
SLOPE cannot fully recover the SLOPE pattern $\beta$, whereas
thresholded SLOPE can recover this pattern.

\begin{figure}[h!]
\centering

\begin{subfigure}{0.49\textwidth}
\centering
\includegraphics[scale=0.55]{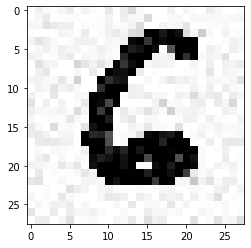}
\subcaption{\label{subfig:six_SLOPE}} 
\end{subfigure}
\begin{subfigure}{0.49\textwidth}
\centering
\includegraphics[scale=0.55]{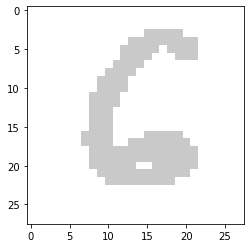}
\subcaption{\label{subfig:six_thres-SLOPE}}
\end{subfigure}

\caption{\label{fig:SLOPE_thres-SLOPE} Pattern recovery by SLOPE
(\subref{subfig:six_SLOPE}) vs thresholded SLOPE
(\subref{subfig:six_thres-SLOPE}): Since the noiseless recovery does
not hold for the particular matrix $X$, the SLOPE estimator
$\hat\beta_{\lambda_\sure}$ is unlikely to recover the SLOPE pattern
$\beta$. Indeed, (\subref{subfig:six_SLOPE}) shows that
$\hat\beta_{\lambda_\sure}$, reshaped as a picture of size $28 \times
28$, does not recover the SLOPE pattern $\beta$. On the other hand,
the accessibility condition does hold and therefore thresholded SLOPE
can reveal the SLOPE pattern $\beta$. In fact,
(\subref{subfig:six_thres-SLOPE}) illustrates that
$\prox_\tau(\hat\beta_{\lambda_\sure})$, reshaped as a picture of size
$28 \times 28$, does recover the SLOPE pattern $\beta$ (here, $\tau >
0$ was chosen as the smallest real number for which
$\prox_\tau(\hat\beta_{\lambda_\sure})$ and $\beta$ have the same
complexity).}

\end{figure}

\section{Conclusion} \label{sec:concl}

This article introduces the concept of patterns as an equivalence
classes for vectors sharing the same subdifferential, also giving a
geometric description of these classes. In view of a linear regression
framework, this article then establishes theoretical properties, based
on an accessibility or noiseless recovery condition, under which a
penalized estimator, or a thresholded variant, recovers the pattern of
the regression coefficients. Our approach offers a unified framework
to provide and interpret conditions that have previously been examined
in the literature within specific methods and in rather technical
settings only.

As a perspective, the notion of patterns could be leveraged to derive
the solution path of a penalized estimator. For instance, the concept
of SLOPE patterns is crucial in developing an algorithm for computing
the solution path for this estimator (\citet{DupuisTardivel24}). This
approach might be extended to other polyhedral gauges beyond the
sorted-$\ell_1$-norm by employing the appropriate pattern. Another
direction is the application of our results to pattern selection. For
the LASSO, \citet{PokarowskiEtAl22} perform subset selection by
constructing a nested family of subsets of the LASSO support and
selecting one via a model selection criterion. This methodology could
be generalized to polyhedral gauges other than the $\ell_1$-norm by
constructing a nested family of patterns.

\section*{Acknowledgments}

We thank Samuel Vaiter, Mathurin Massias, and Abderrahim Jourani for
their insightful comments. Patrick Tardivel was supported by the EIPHI
Graduate School (ANR-17-EURE-0002) and the Bourgogne-Franche-Comté
Region (EPADM project). Tomasz Skalski was supported by a French
Government Scholarship. The work of Piotr Graczyk and Tomasz Skalski
was conducted within the France 2030 program, Centre Henri Lebesgue
(ANR-11-LABX-0020-01). We also thank the anonymous referee for
constructive suggestions that helped to improve the paper.

\appendix

\section{Appdendix -- Facts about polytopes and polyhedral gauges}
\label{app:polytopes}

We recall some basic definitions and facts about polytopes which we
will use throughout the proofs. The following can be found in
textbooks such as \cite{Gruber07} and \cite{Ziegler12}.


A set $P \subseteq \R^p$ is called a \emph{polytope} if it is the
convex hull of a finite set of points $\{v_1,\dots,v_k\} \subseteq \R^p$,
namely,
$$
P = \conv\{v_1,\dots,v_k\}.
$$
The \emph{dimension} $\dim(P)$ of a polytope is defined as the
dimension of $\aff(P)$, the affine subspace spanned by $P$. An
inequality $a'x \leq c$ is called a \emph{valid inequality} of $P$ if
$P \subseteq \{x\in \R^p : a'x \leq c\}$. A \emph{face} $F$ of $P$ is
any subset $F \subseteq P$ that satisfies
$$
F = \{x \in P : a'x = c\} \text{ \rm for some } a \in \R^p \text{ \rm and } c
\in \R \text{ \rm with } P \subseteq \{x\in \R^p : a'x \leq c\}.
$$
Note that $F = \emptyset$ and $F = P$ are faces of $P$ and that any
face $F$ is again a polytope. A non-empty face $F$ with $F \neq P$ is
called \emph{proper}. A point $x_0 \in P$ lies in $\relint(P)$, the
\emph{relative interior} of $P$, if $x_0$ is not contained in a proper
face of $P$. We state two useful properties of faces in the following
lemma.


\begin{lemma} \label{lem:polytope-facts}
Let $P \subseteq \R^p$ be a polytope given by $P =
\conv\{v_1,\dots,v_k\}$, where $v_1,\dots,v_k \in \R^p$. The
following properties hold.

\begin{enumerate}

\item If $F$ and $\tilde F$ are faces of $P$, then so is $F \cap
\tilde F$.

\item Let $L$ be an affine line contained in the affine span of $P$.
If $L \cap \relint(P) \neq \emptyset$, then $L$ intersects a proper
face of $P$.

\end{enumerate}

\end{lemma}

Lemma~\ref{lem:subdiff-faces} characterizes the connection between a
certain class of convex functions (which encompasses polyhedral
gauges) and the faces of a related polytope. The lemma is needed to
prove Theorem~\ref{thm:unique-gauge}.


\begin{lemma} \label{lem:subdiff-faces}
Let $v_1,\dots,v_k\in \R^p$, $P$ be the polytope
$P=\conv\{v_1,\dots,v_k\}$ and $\phi$ be the convex function defined
by
$$
\phi(x) = \max\{v_1'x,\dots,v_k'x\} \text{ \rm for } x \in \R^p.
$$
Then the subdifferential of $\phi$ at $x$ is a face of $P$ and is given by 
$$
\dphi(x) = \conv\{v_l : l \in I_\phi(x)\}  = \{s \in P: s'x= \phi(x)\},\text{ \rm where } 
I_\phi(x) = \{l \in [k] : v_l'x = \phi(x)\}.
$$ 
Conversely, let $F$ be a non-empty face of $P$. Then $F = \dphi(x)$
for some $x \in \R^p$.
\end{lemma}


\begin{proof} The fact that $\dphi(x) = \conv\{u_l : l \in
I_\phi(x)\}$ can be found in \citep[][p. 183]{HiriartLemarechal01}. To
prove the second equality, we consider the following. If $l \in
I_\phi(x)$, by definition of $I_\phi(x)$, $v_l'x = \phi(x)$ and thus
$v_l \in \{s \in P: s'x = \phi(x)\}$. Since $\{s \in P: s'x =
\phi(x)\}$ is a convex set, one may deduce that
$$
\conv\{v_l : l \in I_\phi(x)\} \subseteq \{s \in P : s'x = \phi(x)\}.
$$
Conversely, assume $s \in P$ is such that $s \notin \conv\{v_l : l
\in I_\phi(x)\}$. We then have $s = \sum_{l=1}^k \alpha_l v_l$ where
$\alpha_1,\dots,\alpha_k \geq 0$, $\sum_{l=1}^k \alpha_l = 1$ and
$\alpha_{l_0} > 0$ for some $l_0 \notin I_\phi(x)$. Since $v_l'x \leq
\phi(x)$ for all $l \in [k]$ and $u_{l_0}'x < \phi(x)$, we also get
$$
s'x = \sum_{l=1}^k \alpha_l v_l'x < \phi(x).
$$
Consequently, $s \notin \{s \in P : s'x = \phi(x)\}$ and thus
$$
\{s \in P : s'x = \phi(x)\} \subseteq \conv\{v_l : l \in I_\phi(x)\}.
$$
Therefore, $\dphi(x) = \conv\{v_l : l \in I_\phi(x)\}  = \{s \in P: s'x = \phi(x)\}$.

Now we show that the subdifferentials of $\phi$ are the (non-empty)
faces of $P$. Let $x \in \R^p$. By definition of $\phi$, $v_l'x \leq
\phi(x)$ for every $l \in [k]$ so that the inequality $x's \leq
\phi(x)$ is valid for all $s \in P$. This implies that $\dphi(x)$ is a
non-empty face of $P$. Conversely, let $F = \{s \in P : a's = c\}$ be
a non-empty face of $P$ where $a \in \R^p$, $c \in \R$ and $a's \leq
c$ is a valid inequality for all $s \in P$. We prove that $F =
\dphi(a)$. For this, take any $s \in F$. We get $a's = c$ as well as
$a's \leq \phi(a)$ as shown above, implying that $c \leq \phi(a)$.
Analogously, for any $s \in \dphi(a)$, $a's = \phi(a)$ as well as $a's
\leq c$ since $\dphi(a) \subseteq P$, yielding $\phi(a) \leq c$.
Therefore we may deduce that $\phi(a) = c$ and thus $F = \dphi(a)$.
\end{proof}

\section{Appendix -- Proofs} \label{app:proofs}

We use the following additional notation for the remainder of the
appendix. Given a matrix $A$, $A^+$ stands for the \emph{Moore-Penrose
inverse of $A$}. For a vector $v$, $v^\perp$ represents
$\lin(\{v\})^\perp$, the hyperplane orthogonal to $v$. We denote the
\emph{convex cone} or \emph{positive hull generated by
$v_1,\dots,v_k$} with $\cone(v_1,\dots,v_k)$. The cardinality of an
index set $I$ is denoted by $|I|$.

\subsection{Proofs for Section~\ref{sec:patterns}} 

\subsubsection*{Proof of Theorem~\ref{thm:patt-cones}}

We now prove Theorem~\ref{thm:patt-cones}, the first part stating that
the equivalence classes $C_\beta$ with respect to $\pen$ coincide with
the relative interior of the normal cones of the faces of $B^*$. Note
that, since $B^*$ is a polytope, the normal cones of $B^*$ are the
same at all points in the relative interior of a particular face of
$B^*$, \citep[see e.g.][p.16]{Ewald96}. Since $\dpen(\beta)$ is a face
of $B^*$ by Lemma~\ref{lem:subdiff-faces} in
Appendix~\ref{app:polytopes}, this means that $N_{B^*}(w) =
N_{B^*}(\tilde w)$ for all $w,\tilde w \in \relint(\dpen(\beta))$. For
simplicity, we write $N_{B^*}(\dpen(\beta))$ for the normal cone of
$B^*$ at (any) $b \in \relint(\dpen(\beta))$, so that
Theorem~\ref{thm:patt-cones} states that
$$
C_\beta = \relint(N_{B^*}(\dpen(\beta))).
$$ 

For the second part of Theorem~\ref{thm:patt-cones}, namely
$\lin(C_\beta) = \vec\aff(\dpen(\beta))^\perp$, we will need the
following lemma.

\begin{lemma} \label{lem:normal_cone}
Let $P$ be the polyhedron $\{w \in \R^p: s_1'w \leq r_1,\dots,s_m'w \leq
r_m\}$, $\bar w \in P$ and $\bar I =\{l \in [m]: s_l'\bar w = r_l\}$. We then have
$$
\lin(N_P(\bar w)) = \vec\aff(F)^\perp,
$$
where $F$ is the smallest face of $P$ containing $\bar w$, i.e., $F = \{w \in
P: s_l'w = r_l \quad \forall l \in \bar I\}$.
\end{lemma}

\begin{proof}
According to \citep[Proposition~14.1, p.\ 250]{Gruber07}, we have
$N_P(\bar w) = \cone(\{s_l\}_{l \in \bar I})$ and therefore
$\lin(N_P(\bar w)) = \lin(\{s_l\}_{l \in \bar I})$. Clearly,
$\vec\aff(F) \subseteq \lin(\{s_l\}_{l \in \bar I})^\perp$ and
conversely, if $h \in \lin(\{s_l\}_{l \in \bar I})^\perp$ then, for
$\eta > 0$ small enough, $s_l'(\bar w + \eta h) = r_l$ for all $l \in
\bar I$ and $s_l'(\bar w + \eta h) < r_l$ for all $l \notin \bar I$.
Therefore $\bar w + \eta h \in F$ and thus $h \in \vec\aff(F)$ and
consequently $\lin(N_P(\bar w)) = \lin(\{s_l\}_{l \in \bar I}) =
\vec\aff(F)^\perp$.
\end{proof}

\begin{proof}[Proof of Theorem~\ref{thm:patt-cones}:  $C_\beta = \relint(N_{B^*}(\dpen(\beta)))$] 
We split the proof  into four steps.

\medskip

1) We first show that $C_\beta \subseteq N_{B^*}(\dpen(\beta))$. For
this, take $b \in C_\beta$ and $v \in \relint(\dpen(b))$. Since for
any $z \in B^*$ we have
$$
b'(z - v) = \underbrace{b'z}_{\leq \pen(b)} - \underbrace{b'z}_{ = \pen(b)} 
\leq \pen(b) - \pen(b) = 0,
$$
we may conclude that $b \in N_{B^*}(v) = N_{B^*}(\dpen(b)) =
N_{B^*}(\dpen(\beta))$.

\smallskip

2) Next, we show that $\lin(N_{B^*}(\dpen(\beta))) \subseteq
\vec\aff(\dpen(\beta))^\perp$: take $s \in N_{B^*}(\dpen(\beta))$ and
note that this implies
$$
s'(z - v) \leq 0 \;\; \forall z \in B^*, \forall v \in \relint(\dpen(\beta)).
$$
Since $\dpen(\beta) \subseteq B^*$, we can conclude for any $v,w \in
\relint(\dpen(\beta))$ that both
$$
s'(w - v) \leq 0 \text{ \rm and } s'(v - w) \leq 0
$$
hold so that $s \perp v - w$ for any $v,w \in \relint(\dpen(\beta))$.
Consequently $s \perp v - w$ for any $v,w \in
\dpen(\beta)$\footnote{There exists sequences $(v_n)_{n\in \N}$  and
$(w_n)_{n\in \N}$ in  $\relint(\dpen(\beta))$ such that $\lim_{n \to
\infty} v_n = v$ and  $\lim_{n \to \infty} w_n = w$. Since
$s'(v_n - w_n) = 0$ one may deduce that $s'(v - w) = 0$.}. Therefore, the
following holds.
$$
\lin(N_{B^*}(\dpen(\beta))) \subseteq \lin\{v - w : v,w \in \relint(\dpen(\beta))\}^\perp 
= \vec{\aff}(\dpen(\beta))^\perp.
$$

\smallskip

3) By 1), we have $C_\beta \subseteq N_{B^*}(\dpen(\beta))$. We now
establish the stronger result $C_\beta \subseteq
\relint(N_{B^*}(\dpen(\beta)))$. For this, let $b \in C_\beta$. We
show that $B(b,\eps) \cap \aff(N_{B^*}(\dpen(\beta))) \subseteq
N_{B^*}(\dpen(\beta))$ for small enough $\eps > 0$, implying the
desired claim. Take any $s \in B(b,\eps) \cap
\aff(N_{B^*}(\dpen(\beta)))$. By Lemma~\ref{lem:subdiff-incl}, we know
that $s \in B(b,\eps)$ implies $\dpen(s) \subseteq \dpen(b) =
\dpen(\beta)$ for small enough $\eps > 0$. If $\dpen(s) \subsetneq
\dpen(\beta)$, pick $v \in \dpen(s)$ and $w \in \dpen(\beta) \setminus
\dpen(s)$. Since $v - w \in \vec\aff(\dpen(\beta))$ and $s \in
\aff(N_{B^*}(\dpen(\beta))) \subseteq \lin(N_{B^*}(\dpen(\beta)))$
then, by 2), we have $s \in \vec\aff(\dpen(\beta))^\perp$ and
therefore $s'(v - w) = 0$. Finally, since $s'v = \pen(s)$, we may
deduce that $s'w = \pen(s)$ and thus $w \in \dpen(s)$ which leads to a
contradiction. Consequently, $s \in C_\beta$ so that $B(b,\eps) \cap
\aff(N_{B^*}(\dpen(\beta))) \subseteq  C_\beta \subseteq
N_{B^*}(\dpen(\beta))$.

\smallskip

4) So far, we have shown that $C_\beta \subseteq
\relint(N_{B^*}(\dpen(\beta)))$. We now argue that equality holds. For
this, note that it is known that the relative interior of the normal cones
provide a partition of the underlying space \citep[see
e.g.][p.17]{Ewald96}, so that the sets $\relint(N_{B^*}(\dpen(\beta)))$
form a partition of $\R^p$. Since the sets $C_\beta$ also form a partition
one may deduce that $C_\beta = \relint(N_{B^*}(\dpen(\beta)))$.


\smallskip

We now show the second part $\lin(C_\beta) =
\vec\aff(\dpen(\beta))^\perp$. Because $C_\beta =
\relint(N_{B^*}(\dpen(\beta)))$ and linear subspaces are closed, one
may deduce that $N_{B^*}(\dpen(\beta))) \subseteq
\lin(\relint(N_{B^*}(\dpen(\beta)))$. Consequently,
$\lin(\relint(N_{B^*}(\dpen(\beta))) = \lin(N_{B^*}(\dpen(\beta)))$.
Let $s \in \relint(\dpen(\beta))$. Because $\lin(N_{B^*}(\dpen
(\beta))) = \lin(N_{B^*}(s))$ and since $\dpen(\beta)$ is the smallest
face of $B^*$ containing $s$, we may deduce by
Lemma~\ref{lem:normal_cone} that $\lin(N_{B^*}(s)) = \lin(C_\beta) =
\vec\aff(\dpen(\beta))^\perp$.
\end{proof}

\subsection{Proofs for Section~\ref{sec:recovery}}

\subsubsection*{Proof of Proposition~\ref{prop:acc}}

The following lemma can be seen as generalizing Proposition~4.1 from
\cite{Gilbert17} from the $\ell_1$-norm to all convex functions.

\begin{lemma} \label{lem:equiv-acc} Let $\beta \in \R^p$ and $\phi$ be
a convex function on $\R^p$. Then $\row(X)$ intersects $\dphi(\beta)$
if and only if, for any $b \in \R^p$, the following implication holds
\begin{equation} \label{eq:analytic}
X\beta = Xb \implies \phi(\beta) \leq \phi(b).
\end{equation}
\end{lemma}

\begin{proof}
Consider the function $\iota_\beta : \R^p \to \{0,\infty\}$ given by
$$
\iota_\beta(b) = \begin{cases} 0 & \text{ \rm when } Xb = X\beta \\
\infty & \text{\rm else.} \end{cases}
$$
Then \eqref{eq:analytic} holds for any $b \in \R^p$ if and only if
$\beta$ is a minimizer of the function $b \mapsto \phi(b) +
\iota_\beta(b)$. It is straightforward to show that
$\partial_{\iota_\beta}(\beta) = \row(X)$. We can therefore deduce
that the implication \eqref{eq:analytic} holds for any $b \in \R^p$ if
and only if
$$
0 \in \row(X) + \dphi(\beta) \iff \row(X) \cap \dphi(\beta) \neq
\emptyset.
$$
\end{proof}

\begin{proof}[Proof of Proposition~\ref{prop:acc}] 
By Lemma~\ref{lem:equiv-acc}, the geometric characterization of
accessible patterns is equivalent to the analytic one. We show the
geometric characterization.

($\implies$) When the pattern  of $\beta$ is
accessible with respect to $X$ and $\lambda \pen$, there exists $y \in
\R^n$ and $\hat\beta \in \SpenL(y)$ such that $\hat\beta\overset{\pen}{\sim}
\beta$. Because $\hat\beta$ is a minimizer,
$\frac{1}{\lambda}X'(y - X\hat\beta) \in \dpen(\hat\beta) =
\dpen(\beta)$, so that, clearly, $\row(X)$ intersects $\dpen(\beta)$.

\medskip

($\impliedby$) If $\row(X)$ intersects the face $\dpen(\beta)$, then there
exists $z \in \R^n$ such that $X'z \in \dpen(\beta)$. For $y = X\beta
+ \lambda z$,
we have $\frac{1}{\lambda}X'(y - X\beta)=
X'z$, so that
$\beta \in \SpenL(y)$, and the pattern of $\beta$ is accessible with respect to $X$ and $\lambda\pen$.
\end{proof}

\subsubsection*{Proof of Proposition~\ref{prop:attainability}}

\begin{proof}
Assume that the pattern of $\beta$ is accessible. Using
\citet[Proposition~5.2,(35)]{Gilbert17}\footnote{To make the
connection to the constrained problem treated in this reference, set
$A=X$ and $b=X\hat\beta(y)$.}, we may conclude that there exists $z
\in \R^n$ such that $X'z \in \relint(\dpen(\beta))$. We set $y =
\lambda z + X\beta$ and note that
$$
\frac{1}{\lambda}X'(y - X\beta) = X'z \in \relint(\dpen(\beta)),
$$
so that $y \in A_\beta$. We now show that for small, but otherwise
arbitrary $\eps \in \R^n$, $y + \eps$ still lies in $A_\beta$. For
this, we decompose $\R^n$ into
$$
\R^n = \col(XU_\beta) \oplus \col(XU_\beta)^\perp = \col(XU_\beta) \oplus \ker(U_\beta'X'),
$$
where $U_\beta \in \R^{p \times m}$ contains a basis of
$\lin(C_\beta)$ as columns. (Note that $m$ is the complexity of the
pattern $\beta$.) We accordingly decompose $\eps = \tilde\eps +
\check\eps$, where $\tilde\eps \in \col(XU_\beta)$ and $\check\eps \in
\ker(U_\beta'X')$ which satisfy $\|\check \eps\|_2\le \|\eps\|_2$ and
$\|\tilde \eps\|_2\le \|\eps\|_2$. By construction, we have
$\tilde\eps = XU_\beta(XU_\beta)^+ \tilde \eps$. We set $\tilde\beta =
\beta + U_\beta (XU_\beta)^+ \tilde \eps$. Note that $\beta \in
C_\beta$ and $U_\beta (XU_\beta)^+ \tilde \eps \in \lin(C_\beta)$. By
Theorem~\ref{thm:patt-cones}, $C_\beta$ is relatively open. Moreover,
we have $\lin(C_\beta) = \aff(C_\beta) =\vec\aff(C_\beta)$, which
holds since $0$ lies in the relative boundary of $C_\beta$ by
Theorem~\ref{thm:patt-cones} and $\aff(C_\beta)$ is closed, so that $0
\in \aff(C_\beta)$. Therefore, there exists $r_0>0$ such that
$\|\eps\|_2\le r_0$ implies $\tilde\beta \in C_\beta$. Moreover,
$$
\frac{1}{\lambda}X'(y + \eps - X\tilde\beta) = 
\frac{1}{\lambda}X'\left(y + X U_\beta (XU_\beta)^+ \tilde \eps + \check\eps - X(\beta +  U_\beta (XU_\beta)^+ \tilde \eps)\right) = 
X'z + \frac{1}{\lambda}X'\check\eps.
$$
Since $X'z \in \relint(\dpen(\beta))$ and $X'\check\eps/\lambda \in
\col(U_\beta)^\perp = \lin(C_\beta)^\perp = \vec\aff(\dpen(\beta))$, by
Theorem~\ref{thm:patt-cones}, there exists $r_1>0$ such that  
$\|\eps\|_2\le r_1$ implies  $X'(y + \eps -
X\tilde\beta)/\lambda\in \dpen(\beta)$. Finally, when $\|\eps\|_2\le \min\{r_0,r_1\}$ then  $ \dpen(\tilde\beta)=\dpen(\beta)$ 
proving that $\SpenL(y + \eps) = \{\tilde\beta\}$, where $\tilde\beta
\penequiv \beta$.
\end{proof}

\subsubsection*{Proof of Theorem~\ref{thm:irrep-cond}}

\begin{lemma} \label{lem:subdiff-convex}
Let $\phi : \R^p \to \R$ be the polyhedral gauge defined as 
$$
\phi(x) = \max\{u_1'x,\dots,u_k'x\} \text{ \rm for
some } u_1,\dots,u_k \in \R^p 
$$ 
If $\dphi(x) = \dphi(v)$, we have $\dphi(x) = \dphi(\alpha x + (1 -
\alpha)v) = \dphi(v)$ for all $\alpha \in [0,1]$.
\end{lemma}


\begin{proof}
Let $s \in \dphi(x) = \dphi(v)$. Since $s$ is a subgradient at $x$ and
at $v$, the following two inequalities hold
\begin{align*}
\phi(\alpha x + (1 - \alpha)v) &\geq \phi(x) - (1 - \alpha)s'(x-v) \\
\phi(\alpha x + (1 - \alpha)v) &\geq \phi(v) + \alpha s'(x-v).
\end{align*}
Multiplying the first inequality by $\alpha$, the second by
$(1-\alpha)$ and adding them, we get
$$
\phi(\alpha x + (1-\alpha)v) \geq \alpha\phi(x) + (1-\alpha)\phi(v).
$$
Using the convexity of $\phi$, we arrive at
$$
\phi(\alpha x + (1-\alpha)v) = \alpha\phi(x) + (1 - \alpha)\phi(v).
$$
By Lemma~\ref{lem:subdiff-faces} we have $\dphi(x) = \conv\{u_l : l \in I\}$, where $I_\phi(x) = \{l \in [k] :
u_l'x = \phi(x)\}$. Therefore, if
$u_l \in \dphi(x) = \dphi(v)$, then $u_l'x = \phi(x)$ and $u_l'v =
\phi(v)$, thus
$$
u_l'(\alpha x + (1 - \alpha)v) = \alpha\phi(x) + (1 - \alpha)\phi(v) = 
\phi(\alpha x + (1 - \alpha)v).
$$
Consequently, $u_l \in \dphi(\alpha x + (1 - \alpha)v)$. On the other
hand, if $u_l \notin \dphi(x)$, then $u_l'x < \phi(x)$ and $u_l'v <\phi(v)$, thus
$$
u_l'(\alpha x + (1 - \alpha)v) < \alpha\phi(x) + (1 - \alpha)\phi(v) = 
\phi(\alpha x + (1 - \alpha)v).
$$
Consequently, $u_l \notin \dphi(\alpha x + (1 - \alpha)v)$ and the
claim follows.
\end{proof}

 
\begin{lemma}\label{lem:Vbeta}
Let $X \in \R^{n \times p}$ and $\beta \in \R^p$. The following set is convex
$$
V_\beta = \{y\in \R^n: \exists \lambda > 0\, \exists \hat\beta \in
\SpenL(y) \text{ \rm such that } \hat \beta \overset{\pen}{\sim} \beta\}.
$$
\end{lemma}
\noindent
Note that $V_\beta$ may be empty.


\begin{proof}
Assume that $V_\beta \neq \emptyset$. Let $y$, $\tilde y \in V_\beta$.
Then there exist $\lambda > 0$ and $\tilde\lambda > 0$ such that
$\hat\beta \in \SpenL(y)$ and $\tilde\beta \in
S_{X,\tilde\lambda\pen}(\tilde y)$ with $\dpen(\hat\beta) =
\dpen(\tilde\beta) = \dpen(\beta)$. Consequently,
$$
X'(y - X\hat\beta) \in \lambda \dpen(\beta) \text{ \rm and }
X'(\tilde y - X\tilde\beta) \in \tilde\lambda\dpen(\beta).
$$
Let $\alpha \in (0,1)$ and $\check y = \alpha y + (1 - \alpha)\tilde
y$. Define $\check\lambda = \alpha\lambda + (1 - \alpha)\tilde\lambda$
and $\check\beta = \alpha\hat\beta + (1 - \alpha)\tilde\beta$.
We show that $\check y \in V_\beta$. Indeed, observe that
$$ 
X'\left(\check y - X\check\beta\right) = \alpha X'(y - X\hat\beta) +
(1-\alpha)X'(\tilde y-X\tilde\beta) \in \alpha\lambda\dpen(\beta) + 
(1 - \alpha)\tilde\lambda\dpen(\beta) = \check\lambda\dpen(\beta).
$$
By Lemma~\ref{lem:subdiff-convex}, $\dpen(\check\beta) =
\dpen(\alpha\hat\beta + (1-\alpha)\tilde\beta) = \dpen(\beta)$, so
that $\check\beta \in S_{X,\check\lambda\pen}(\check y)$ also, which
proves the claim.
\end{proof}


\begin{proof}[Proof of Theorem~\ref{thm:irrep-cond}]
Assume that the noiseless recovery condition does not hold for
$\beta$. Then $X\beta \notin V_\beta$, where $V_{\beta}$ is defined as
in Lemma~\ref{lem:Vbeta}. Consequently, by convexity of $V_\beta$, we
have $X\beta + \eps \notin V_\beta$ or $X\beta - \eps \notin V_\beta$
for any realization of $\eps \in \R^n$. Therefore
\begin{align*}
1 &= \P_\eps(\{X\beta + \eps \notin  V_\beta\} \cup \{X\beta - \eps \notin V_\beta\}) \\
&\leq \P_\eps(\{X\beta + \eps \notin V_\beta\}) + \P(\{X\beta - \eps \notin V_\beta\})
= 2\P_\eps(\{X\beta + \eps \notin  V_\beta\}).
\end{align*}
Consequently, 
$$
\frac{1}{2} \geq \P_\eps(\{X\beta + \eps \in V_\beta\}) = 
\P_\eps(\exists \lambda > 0 \; \exists \hat\beta \in \SpenL(Y) 
\text{ \rm such that } \hat \beta \overset{\pen}{\sim} \beta).
$$
\end{proof}

\subsubsection*{Proof of Proposition~\ref{prop:nrc-geom}}

\begin{proof}
($\implies$) Let $\hat\beta \in \SpenL(X\beta)$ with $ \hat\beta \penequiv \beta$. Then
$$
\frac{1}{\lambda}X'X(\beta - \hat\beta) \in \dpen(\beta). 
$$
Since $\hat\beta \in C_\beta$, we get $(\beta - \hat\beta)/\lambda \in
\lin(C_\beta)$ which yields the desired implication.

\smallskip

($\impliedby$) We assume that $X'X\lin(C_\beta) \cap \dpen(\beta) \neq
\emptyset$, i.e., there exists such $b \in \lin(C_\beta)$ such that
$X'Xb \in \dpen(\beta)$. We set $\hat\beta = \beta - \lambda b$.
By Theorem~\ref{thm:patt-cones}, $C_\beta$ is relatively open.
Note that $b \in \lin(C_\beta) = \aff(C_\beta) = \vec\aff(C_\beta)$,
which holds since $0$ lies in the relative boundary of $C_\beta$ by
Theorem~\ref{thm:patt-cones} and $\aff(C_\beta)$ is closed, so that $0
\in \aff(C_\beta)$. Therefore, for small enough $\lambda$, we have
$\hat\beta \in C_\beta$ and $\hat\beta \penequiv \beta$. Consequently,
$$
\frac{1}{\lambda}X'(X\beta - X\hat\beta) = X'Xb \in \dpen(\hat\beta),
$$ 
so that $\hat\beta \in \SpenL(X\beta)$, which finishes the proof.
\end{proof}

\subsection{Proofs for Section~\ref{sec:recovery-thres}}

\subsubsection*{Proof of Theorem~\ref{thm:overfit}}

Lemmas~\ref{lem:K-compact} and \ref{lem:conv} are used to prove
Theorem~\ref{thm:overfit} which claims that, asymptotically,
 $\hat\beta^{(r)}$ (in the notation of Theorem~\ref{thm:overfit})
converges to $\beta$ when $r$ tends to $\infty$.

\begin{lemma} \label{lem:K-compact}
Let $\pen$ be a real-valued polyhedral gauge on $\R^p$, $X \in \R^{n
\times p}$, $v \in \col(X)$. Let $K_1 \geq 0$, $K_2 \geq 0$  be large
enough such that $K= \{b \in \R^p: \pen(b) \leq K_1, \|Xb - v\|_2 \leq
K_2\}$ is non-empty. If $\{b\in \R^p:Xb=0 \text{ \rm and }\pen(b)=0\} = \{0\}$, the set $K$
is compact.
\end{lemma}

\begin{proof} Clearly, $K$ is closed. Assume that $K$ is unbounded. Then 
there exists an unbounded sequence $(b_r)_{r \in \N} \subseteq
K$ with $\lim_{r\to \infty}\|b_r\|_2 = \infty$. Let $\tilde b_r = b_r/\|b_r\|_2$ and pick a convergent subsequence
$\tilde b_{r_l} \to \tilde b \neq 0$ as $l \to \infty$. We have
$\pen(\tilde b_{r_l}) = \pen(b_{r_l})/\|b_{r_l}\|_2 \leq
K_1/\|b_{r_l}\|_2 \to 0$, so that by continuity, we get $\pen(\tilde
b) = 0$. Similarly, since $(Xb_{r_l})_{l \in \N}$ is bounded, $X\tilde
b_{r_l} = Xb_{r_l}/\|b_{r_l}\|_2 \to 0$, implying that $X\tilde b =
0$. But then $\tilde b = 0$ must hold which yields a contradiction.
\end{proof}

\begin{lemma} \label{lem:conv}
Let $X \in \R^{n\times p}$ and $\pen$ be a real-valued polyhedral
gauge on $\R^p$. Let $\beta \in \R^p$ and set $y^{(r)} = X\beta +
\eps^{(r)}$ where $\eps^{(r)}$ is a sequence in $\R^n$ such that
$\lim_{r\to \infty} \eps^{(r)} = 0$. Assume that uniform uniqueness
holds for $X$ and $\pen$, so that $\hat\beta^{(r)}$ is the unique
element of $\SpenLr(y^{(r)})$. If $\beta$ is accessible with respect
to $X$ and $\pen$ and $\lim_{r \to \infty} \lambda^{(r)} = 0$, we have
$$
\lim_{r \to \infty} \hat\beta^{(r)} = \beta.
$$
\end{lemma}

\begin{proof}
Since $\hat\beta^{(r)}$ is a minimizer, we have 
$$
\frac{1}{2}\|y^{(r)} - X\hat\beta^{(r)}\|^2_2 + \lambda^{(r)}\pen(\hat\beta^{(r)})
\leq \frac{1}{2}\|y^{(r)} - X(\beta+X^ + \eps^{(r)})\|^2_2 +
\lambda^{(r)}\pen(\beta + X^+\eps^{(r)}). 
$$
Because $XX^+$ is the orthogonal projection onto $\col(X)$, we have
$$
\|y^{(r)} - X\hat\beta^{(r)}\|^2_2 \geq \|y^{(r)} - XX^+y^{(r)}\|^2_2 = 
\|y^{(r)} - X(\beta+X^+\eps^{(r)})\|^2_2.
$$ 
so that $\pen(\hat\beta^{(r)}) \leq \pen(\beta+X^+\eps^{(r)})$ and
\begin{equation} \label{eq:bounded} 
\limsup_{r \to \infty} \pen(\hat\beta^{(r)}) \leq \limsup_{r \to \infty} \pen(\beta+X^+\eps^{(r)}) 
= \pen(\beta),
\end{equation}
implying that the sequence $\left(\pen(\hat\beta^{(r)})\right)_{r \in
\N}$ is bounded. Moreover,
$$
\frac{1}{2}\|y^{(r)}  - X\hat\beta^{(r)}\|^2_2 \leq 
\frac{1}{2}\|y^{(r)} - X(\beta+X^+\eps^{(r)})\|^2_2 + \lambda^{(r)}\pen(\beta+X^+\eps^{(r)}). 
$$
The right-hand side of the above display converges to 0 as $r \to
\infty$, so that we can deduce that
$$
\lim_{r \to \infty}\|X\beta - X\hat\beta^{(r)}\|_2 = 0.
$$
Due to uniform uniqueness, we have $\{b\in \R^p:Xb=0 \text{ \rm and }
\pen(b)=0\} = \{0\}$ and thus, by Lemma~\ref{lem:K-compact}, the
sequence $(\hat\beta^{(r)})_{r \in \N}$ is bounded. Therefore, to
prove that $\lim_{r \to \infty} \hat\beta^{(r)} = \beta$, it suffices
to show that $\beta$ is the unique accumulation point of this
sequence. We extract a subsequence $(\hat\beta^{\phi(r)})_{r \in \N}$
converging to $\gamma\in \R^p$ (where $\phi: \N \to \N$ is an
increasing function). By \eqref{eq:bounded}, one may deduce that
$\pen(\gamma) \leq \pen(\beta)$. Moreover, we have that
$$
0 = \lim_{r \to \infty} 
\left\|X\left(\hat\beta^{\phi(r)} - \beta\right)\right\|^2_2
= \|X(\gamma-\beta)\|_2^2.
$$
Therefore, $\gamma$ satisfies
$$
X\gamma = X\beta \textrm{ and } \pen(\gamma) \leq \pen(\beta).
$$
Since the pattern of $\beta$ is accessible, this implies that
$\pen(\gamma) = \pen(\beta)$ by Proposition~\ref{prop:acc}. Using the
same proposition, we can also pick $X'z \in \dpen(\beta)$ and define
$y = X\beta + z$. Since for this choice of $y$, $X'(y-X\beta) = X'z
\in \dpen(\beta)$, we have $\beta \in \Spen(y)$. Since $X\gamma =
X\beta$ and $\pen(\gamma) = \pen(\beta)$, $\gamma \in \Spen(y)$ also.
Uniform uniqueness now implies that $\gamma = \beta$ must hold.
\end{proof}

\begin{lemma} \label{lem:subdiff-incl}
Let $\pen$ be a real-valued polyhedral gauge on $\R^p$ and let $\beta
\in \R^p$. Then there exists $\tau > 0$ such that
$$
\dpen(b) \subseteq \dpen(\beta) \textrm{ for all } b: \|b - \beta\|_\infty \leq \tau.
$$
\end{lemma}

\begin{proof}
Let $I = \{l \in [k]: u_l'\beta = \pen(\beta)\}$. By
Lemma~\ref{lem:subdiff-faces}, $\dpen(\beta) = \conv\{u_l\}_{l \in
I}$. Since
$$
u_l'\beta < \pen(\beta) \; \forall l \notin I, 
$$
and by continuity, we can pick $\tau > 0$ small enough such that
$$
u_l'b < \pen(b) \; \forall l \notin I, \; \forall b: \|b - \beta\|_\infty \leq \tau.
$$
Consequently, for any such $b$, we have $J = \{l \in [k]: u_l'b =
\pen(b)\} \subseteq I$ and thus
$$
\dpen(b) = \conv\{u_l\}_{l \in J} \subseteq \conv\{u_l\}_{l \in I} = \dpen(\beta).
$$
\end{proof}

Finally, the proof of Theorem~\ref{thm:overfit} is based on
Lemma~\ref{lem:conv} and on Lemma~\ref{lem:subdiff-incl} is given
below.

\begin{proof}[Proof of Theorem~\ref{thm:overfit}]
By Lemma~\ref{lem:subdiff-incl}, there exists $\tau > 0$ such that for
any $b$ with $\|b - \beta\|_\infty \leq \tau$, we have $\dpen(b)
\subseteq \dpen(\beta)$. By Lemma~\ref{lem:conv}, $\hat\beta^{(r)}$
converges to $\beta$ when $r$ tends to $\infty$. Consequently, we have
$$
\exists r_0 \in \N \text{ \rm such that } \forall r \geq r_0,  \dpen(\hat\beta^{(r)}) \subseteq \dpen(\beta).
$$
\end{proof}

\subsubsection*{Proof of Proposition~\ref{prop:acc-thres}}

\begin{proof} 
We only need to prove the implication ($\impliedby$), as the other
direction is obvious. Assume that $\dpen(\hat\beta) \subseteq
\dpen(\beta)$. Since $\hat\beta \in \SpenL(y)$, we have
$\frac{1}{\lambda}X'(y - X\hat\beta) \in \dpen(\hat\beta) \subseteq
\dpen(\beta)$. Consequently, $\row(X)$ intersects $\dpen(\beta)$ which
implies that the pattern of $\beta$ is accessible with respect to $X$
and $\pen$ by Proposition~\ref{prop:acc}. Consequently, there exists
$y \in \R^n$ and there exists $\hat\beta \in \SpenL(y)$ for which
$\hat \beta \overset{\pen}{\sim} \beta$.
\end{proof}

\subsection{Proofs for Section~\ref{sec:unique}}

\subsubsection*{Proof of Theorem~\ref{thm:unique-gauge}}

The following lemma -- needed to show Theorem~\ref{thm:unique-gauge}
-- states that the fitted values are unique over all non-unique
solutions of the penalized problem for a given $y$. It is a
generalization of Lemma~1 in \cite{Tibshirani13}, which shows this
fact for the special case of the LASSO.


\begin{lemma} \label{lem:fitted-values}
Let $X \in \R^{n \times p}$, $y \in \R^n$, $\lambda > 0$ and $\pen$ be
a polyhedral gauge. Then $X\hat\beta = X\tilde\beta$ and $\pen(\hat
\beta) = \pen(\tilde\beta)$  for all $\hat\beta, \tilde\beta \in
\Spen(y)$.
\end{lemma}


\begin{proof}
Assume that $X\hat\beta \neq X\tilde\beta$ for some $\hat\beta,
\tilde\beta \in \SpenL(y)$ and let $\check\beta = (\hat\beta +
\tilde\beta)/2$. Because the function $\mu \in \R^n \mapsto
\|y-\mu\|_2^2$ is strictly convex, one may deduce that
$$
\|y - X\check\beta\|_2^2 < 
\frac{1}{2}\|y - X\hat\beta\|_2^2 + \frac{1}{2}\|y - X\tilde{\beta}\|_2^2.
$$ 
Consequently,
$$
\frac{1}{2}\|y - X\check\beta\|_2^2 + \lambda\pen(\check\beta) < 
\frac{1}{2}\left(\frac{1}{2}\|y - X\hat \beta\|_2^2 + \lambda\pen(\hat\beta ) 
+ \frac{1}{2}\|y - X\tilde\beta\|_2^2 + \lambda\pen (\tilde\beta)\right),
$$
which contradicts both $\hat \beta$ and $\tilde \beta$ being
minimizers. Finally, $X\hat\beta = X\tilde\beta$ clearly implies
$\pen(\hat\beta) = \pen(\tilde\beta)$.
\end{proof}


\begin{proof}[Proof of Theorem~\ref{thm:unique-gauge}]
($\implies$) Assume that there exists a face $F$ of $B^* =
\conv\{u_1,\dots,u_k\}$ that intersects $\row(X)$ and satisfies
$\dim(F) < \defect(X)$. By Lemma~\ref{lem:subdiff-faces}, $F =
\dpen(\hat\beta)$ for some $\hat\beta \in \R^p$. Let $z \in \R^n$ with
$X'z \in F$, which exists by assumption. Now let $y = X\hat\beta +
\lambda z$. Note that $\hat\beta \in \SpenL(y)$ since
$$
0 \in X'X\hat\beta - X'y +
\lambda\dpen(\hat\beta) \iff 
\frac{1}{\lambda}X'(y - X\hat\beta) = X'z \in \dpen(\hat\beta).
$$
We now construct $\tilde\beta \in \SpenL(y)$ with $\tilde\beta \neq
\hat\beta$. According to Lemma~\ref{lem:subdiff-faces},
$\dpen(\hat\beta) = \conv\{u_l : l \in I\}$ where $I =
I_\pen(\hat\beta) = \{l \in [k]: u_l'\hat\beta = \pen(\hat\beta)\}$
and thus $u_l'\hat \beta < \pen(\hat\beta)$ whenever $l \notin I$. We
now show that it is possible to pick $h \in \ker(X)$ with $h \neq 0$
but $u_l'h = 0$ for all $l \in I$. We then make $h$ small enough such
that $u_l'(\hat\beta + h) \leq \pen(\hat\beta)$ still holds for all $l
\notin I$, which in turn implies that $\pen(\hat\beta + h) =
\max\{u_l'\hat\beta : l \in I\} = \pen(\hat\beta)$. This, together
with $X\hat\beta = X(\hat\beta + h)$, yields $\hat\beta \neq
\tilde\beta = \hat\beta + h \in \SpenL(y)$ also. We now show that
$\ker(X) \cap \col(U)^\perp \neq \{0\}$, where $U = (u_l)_{l \in I}
\in \R^{p \times |I|}$. For this, we distinguish two cases:


1) Assume that $0 \in \aff\{u_l : l \in I\}$. Then $\aff\{u_l : l \in
I\} = \col(U)$ and $\dim(F) = \rk(U) < \defect(X)$. This implies
that 
$$
\dim(\ker(X)) + \dim(\col(U)^\perp) > p,
$$ which proves what was claimed.


2) Assume that $0 \notin \aff\{u_l : l \in I\}$.  Note that this implies
that $v = X'z \in \row(X) \cap \conv\{u_l : l \in I\}$ satisfies $X'z \neq
0$. We also have $\rk(U) = \dim(\aff\{u_l : l \in I\}) + 1 = \dim(F) +
1 \leq \defect(X)$ which implies that 
$$
\dim(\ker(X)) + \dim(\col(U)^\perp) \geq p.
$$
If $\ker(X) \cap \col(U)^\perp = \{0\}$, then $\R^p = \ker(X) \oplus
\col(U)^\perp$. But we also have $\ker(X) \subseteq v^\perp$ as well as 
$\col(U)^\perp \subseteq v^\perp$, yielding a contradiction and proving the claim.

\medskip


\noindent ($\impliedby$) Assume that there exists $y \in \R^n$ such that
$\hat\beta, \tilde\beta \in \SpenL(y)$ with $\hat\beta \neq
\tilde\beta$. We then have
$$
\frac{1}{\lambda}X'(y-X\hat\beta) \in \dpen(\hat\beta) \;\; \text{ \rm and } \;\; 
\frac{1}{\lambda} X'(y-X\tilde\beta) \in \dpen(\tilde{\beta}).
$$
According to Lemma~\ref{lem:fitted-values}, $X\hat\beta =
X\tilde\beta$, thus $\frac{1}{\lambda}X'(y-X\hat\beta) =
\frac{1}{\lambda}X'(y-X\tilde\beta)$. Consequently, one may deduce
that $\row(X)$ intersects the face $\dpen(\hat\beta) \cap
\dpen(\tilde\beta)$. Let $F^* = \conv\{u_l : l \in I^*\}$ be a face of
$\dpen(\hat\beta) \cap \dpen(\tilde\beta)$ of smallest dimension among
all faces of $\dpen(\hat\beta) \cap \dpen(\tilde\beta)$ intersecting
$\row(X)$. By minimality of $\dim(F^*)$, $\row(X)$ intersects the
relative interior of $F^*$, namely, there exists $z \in \R^n$ such
that $v = X'z$ lies in $F^*$, but not on a proper face of $F^*$. We
will now show that if $\dim(F^*) \geq \defect(X)$, then $\row (X)$
intersects a proper face of $F^*$, yielding a contradiction.


For this, first observe that $\dim(F^*) = \dim(\aff\{u_l : l \in
I^*\})$ and that we can write the affine space $\aff\{u_l : l \in
I^*\} = u_{l_0} + \col(\tilde U^*)$ where $l_0 \in I^*$ and $\tilde
U^* = (u_l \!-\! u_{l_0})_{l \in I^*\setminus\{l_0\}} \in \R^{p \times
|I^*| - 1}$, implying that $\dim(F^*) = \rk(\tilde U^*)$.


Now let $h = \hat\beta - \tilde\beta \neq 0$. Clearly, $h \in
\ker(X)$. Moreover, since $\pen(\hat\beta) =
\pen(\tilde\beta)$ by Lemma~\ref{lem:fitted-values}, and since $u_l \in \dpen(\hat\beta) \cap
\dpen(\tilde\beta)$ for all $l \in I^*$, by
Lemma~\ref{lem:subdiff-faces}, we get
$$
u_l'h = u_l'\hat\beta - u_l'\tilde\beta = \pen(\hat\beta) -
\pen(\tilde \beta) = 0 \; \forall l \in I^*.
$$
Therefore, $h \in \ker(X) \cap \col(U^*)^\perp$, where $U^* = (u_l)_{l
\in I^*} \in \R^{p \times |I^*|}$. Assume that $\dim(F^*) \geq
\defect(X)$. Then
$$
\dim(\row(X)) + \dim(\col(\tilde U^*)) \geq \rk(X) + \defect(X) = p.
$$
If $\row(X) \cap \col(\tilde U^*) = \{0\}$, then $\R^p = \row(X)
\oplus \col(\tilde U^*)$. However, the last relationship cannot hold
since $\row(X) = \ker(X)^\perp \subseteq h^\perp$ as well as
$\col(\tilde U^*) \subseteq \col(U^*) \subseteq h^\perp$, where $h
\neq 0$. Consequently, there exists $0 \neq \tilde v \in \row(X) \cap
\col(\tilde U^*)$. The affine line $L = \{X'z + t\tilde v : t \in \R\}
\subseteq \row(X)$ intersects the relative interior of $F^*$ at $t =
0$ and clearly lies in $\aff(F^*) = u_{l_0} + \col(\tilde U^*)$, since
$X'z \in F^*$ and $\tilde v \in \col(\tilde U^*)$. Therefore, $L$ must
intersect a proper face of $F^*$ by Lemma~\ref{lem:polytope-facts}.
But then also $\row(X)$ intersects a proper face of $F^*$, which
yields the required contradiction.
\end{proof}

\subsubsection*{Proof of Proposition~\ref{prop:unique-sym_gauge}}

Before turning to Proposition~\ref{prop:unique-sym_gauge}, we prove
the following lemma.

\begin{lemma} \label{lem:proper_face} 
Let $\pen$ be a symmetric polyhedral gauge defined by $\pen(b) =
\max\{\pm u_1'b,\dots,\pm u_k'b\}$ for some $u_1,\dots,u_k \in \R^p$.
Let $F$ be a face of $B^* = \conv\{\pm u_1,\dots,\pm u_k\}$. If $0 \in
F$ then $ F = B^*$.
\end{lemma}

\begin{proof}
According to Lemma~\ref{lem:subdiff-faces}, there exists $b \in \R^p$
such that $F = \dpen(b) = \{s \in B^* : s'b = \pen(b)\}$. Since $0 \in
F$, $\pen(b) = 0$ and by symmetry, $u_l'b  = 0 = \pen(b)$ for all $l
\in [p]$. The latter already implies that $\pm u_l \in F$ for all $l \in
[p]$ and therefore, using the convexity of $F$, we get
$B^* = \conv\{\pm u_l : l \in [p]\} \subseteq F$. Consequently, $F = B^*$ must hold.
\end{proof}

\begin{proof}[Proof of Proposition~\ref{prop:unique-sym_gauge}]
Since $\pen$ is symmetric, we have
$$
\pen(b)=0 \iff u_l'b = 0 \; \forall l \in [k] \iff b \in \col(U)^\perp,
$$
where $U = (u_1, \dots, u_k) \in \R^{p \times k}$, so that clearly,
$\ker(\pen)$ is a vector space.

\ref{enum:X}. Let $X \in \R^{n \times p}$. We have $\defect(X) \geq p
- n$, so that $\dim(\{b\in \R^p:\pen(b)=0\}) > n$ implies that
$\ker(\pen) \cap \ker(X) \neq \{0\}$. Consequently, for any $\hat\beta
\in \SpenL(y)$ for arbitrary $y \in \R^n$ and $\lambda > 0$, we have
$\hat\beta + (\ker(\pen) \cap \ker(X)) \subseteq \SpenL(y)$.
Therefore, $\SpenL(y)$ is unbounded.

\medskip

\ref{enum:mu}. The Lebesgue measure on $\R^{n \times p}$ and the
standard Gaussian measure on $\R^{n \times p}$ are equivalent. Let $Z$
be a random matrix in $\R^{n \times p}$ having iid entries following a
standard normal distribution. It therefore suffices to show that the
event $\{\exists \text{ a face } F \text{ of } B^* : \row(Z) \cap F
\neq \emptyset, \dim(F) < \defect(Z)\}$ has zero probability. Note
that for $n \geq p$, $\ker(Z) = \{0\}$ almost surely and the
probability of the above event is indeed equal to zero.

Now assume that $n < p$. Note that $\row(Z)$ trivially intersects  $
B^* = \conv\{\pm u_1,\dots,\pm u_k\}$ at $0$. We first prove that
$\P_Z(\dim(B^*) < \defect(Z)) = 0$.  Since $0 \in B^*$, we have
$\dim(B^*) = \rk(U)$. As $\dim(\col(U)^\perp) \leq n$ by assumption,
we have $\dim(B^*) \geq p - n$. Moreover, $\defect(Z) = p - n$ almost
surely, implying that $\P_Z(\dim(B^*) < \defect(Z)) = 0$. Now, let $F$
be a proper face of $B^*$ for which $\dim(F) < p - n$. We prove that
$\P_Z \left(\row(Z) \cap F \neq \emptyset \right) = 0$. Recall that if
$V = (V_1,\dots,V_n) \in \R^{q \times n}$ has iid $\mN(0,1)$ entries,
then $\P_V(v \in \col(V)) = 0$, where $q \geq n+1$  and $v \neq 0$,
see Lemma~20 in~\cite{SchneiderTardivel22}. Note that $\codim(F) > n$
and $0 \notin \aff(F)$. Indeed, $0 \in B^*$ and according to
Lemma~\ref{lem:proper_face}, a proper face of $B^*$ does not contain
the origin thus $0 \notin F$. Since $F = \aff(F)\cap B^*$ one may
deduce $0 \notin \aff(F)$ also. There exists $A \in \R^{q \times p}$
with $q = \codim(F)$ and orthonormal rows, as well as $v \in \R^q$, $v
\neq 0$ such that $\aff(F) = \{b \in \R^p: Ab = v\}$. Since $AA' =
\I_q$, $AZ' \in \R^{q \times n}$ has iid $\mN(0,1)$ entries, we have
\begin{equation} \label{eq:negligible}
\P_Z \left(\row(Z) \cap F \neq \emptyset \right) \leq 
\P_Z(\row(Z) \cap \aff(F) \neq \emptyset) = \P_Z(v \in \col(AZ')) = 0.  
\end{equation}
Let $\mF(B^*)$ denote the set of faces of the polytope $B^*$ which is
finite. Using Theorem~\ref{thm:unique-gauge} and since $\defect(Z) = p
- n$ almost surely, the following equalities hold.
\begin{align*}
\P_Z(\exists y \in \R^n, |\SpenZL(y)| > 1) & =
\P_Z\bigg(\bigcup_{\underset{\dim(F) < \defect(Z)}{F \in
\mF(B^*)}}\{\row(Z) \cap F \neq \emptyset\}\bigg) \\ 
& = \P_Z\bigg(\bigcup_{\underset{\dim(F) <p-n}{F \in \mF(B^*)}}\{\row(Z)
\cap F\neq \emptyset\}\bigg) = 0,
\end{align*}
where the last equality is a consequence of~\eqref{eq:negligible}.
\end{proof}

\section{Appendix -- Additional results} \label{app:add-results}

\subsection{Existence of a minimizer} \label{subapp:existence}

We show that the optimization problem of interest in this article
always has a minimizer.

\begin{proposition} \label{prop:existence}
Let $X \in \R^{n\times p}$, $y \in \R^n$, $\pen(x) =
\max\{u_1'x,\dots,u_l'x\}$ where $u_1,\dots,u_l \in \R^p$ with $u_1 =
0$. For 
$$
f: b \in \R^p \mapsto \frac{1}{2} \|y - Xb\|_2^2 + \lambda\pen(b),
$$
the optimization problem 
$
\min_{b\in \R^p} f(b)
$
has at least one minimizer. 
\end{proposition}

For the remainder of this section, without loss of generality, we set $\lambda = 1$ since
otherwise, this parameter can be absorbed into the penalty function.
The proof of Proposition \ref{prop:existence} relies on the following
two lemmas.

\begin{lemma} \label{lem:min_sequence}
Let the assumptions of Proposition~\ref{prop:existence} hold and let
$(\beta_m)_{m \in \N}$ be a minimizing sequence of $f$:
$$
\lim_{m \to \infty} f(\beta_m) = \inf_{b \in \R^p} f(b).
$$
Then also $(X\beta_m)_{m \in \N}$ and $(\pen(\beta_m))_{m \in \N}$ converge.
Moreover, these limits do not depend on the minimizing sequence.
\end{lemma}

\begin{proof}
The sequence $(X\beta_m)_{n\in \N}$ is bounded. Otherwise, $\|y -
X\beta_m\|_2^2$ would be unbounded also, contradicting $\inf\{f(b) :
b\in \R^p\} \leq f(0) < \infty$. Let $\tilde\beta_m$ be another
minimizing sequence. Note that also $X\tilde\beta_m$ is bounded. Now
extract arbitrary converging subsequences $(X\beta_{n_m})_{m \in \N}$
and $(X\tilde\beta_{\tilde n_m})_{m \in \N}$ with limits $l$ and
$\tilde l$, respectively. Note that $(\beta_{n_m})_{m \in \N}$ and
$(\tilde\beta_{\tilde n_m})_{m \in \N}$ are still minimizing sequences
so that also $\pen(\beta_{n_m})$ and $\pen(\tilde \beta_{\tilde n_m})$
must converge. We now show that $l = \tilde l$. If $l \neq \tilde l$,
set $\bar\beta_m = (\beta_{n_m} + \tilde\beta_{\tilde n_m})/2$. By the
above considerations, $(f(\bar\beta_m))_{m \in \N}$ is convergent.
Since the function $z \in \R^n \mapsto \|y - z\|_2^2$ is strictly
convex and $\pen$ is convex, we may deduce that
\begin{align*}
\limsup_{m \to \infty} f(\bar\beta_m) & \leq \frac{1}{2} \|y - (l + \tilde l)/2\|_2^2 
+ \limsup_{m \to \infty} \pen(\bar\beta_m) \\
& < \frac{1}{2}\left(\|y-l\|_2^2/2 + \|y-\tilde l\|^2_2/2\right) 
+ \lim_{m \to \infty}\pen(\beta_{n_m})/2 + \lim_{m \to \infty}\pen(\tilde\beta_{\tilde n_m})/2 \\
& = \frac{1}{2}\lim_{m \to \infty} f(\beta_{n_m}) + \frac{1}{2}\lim_{m \to \infty}f(\tilde\beta_{\tilde n_m})
= \inf_{b \in \R^p} f(b),
\end{align*}
yielding a contradiction. Since the selection of convergent
subsequences was arbitrary, this implies that $(X\beta_m)_{m \in \N}$
and $(X\tilde\beta_m)_{m \in \N}$ share a unique limit point and that
the sequences $(\pen(\beta_m))_{m \in \N}$ and
$(\pen(\tilde\beta_m))_{m \in \N}$ converges as well.
\end{proof}
We remark that Lemma~\ref{lem:min_sequence} also holds for any
non-negative, convex function in place of the polyhedral gauge $\pen$.

\begin{lemma} \label{lem:constrained}
Let the assumptions of Proposition~\ref{prop:existence} hold and let
$\gamma \geq 0$. The optimization problem
\begin{equation} \label{eq:constrained}
\min_{b \in \R^p} \|y - Xb\|^2_2 \;\; \text{ \rm subject to } \;\;
\pen(b) \leq \gamma
\end{equation}
has at least one minimizer. 
\end{lemma}

\begin{proof}
Let $P_\gamma = \{b \in \R^p : \pen(b) \leq \gamma\}$ be the closed
and convex feasible region of \eqref{eq:constrained}. We set $z = Xb$
and note that the linearly transformed set $XP_\gamma$ is still closed
and convex. Therefore, the minimization problem
$$
\min \; \|y-z\|^2_2 \;\; \text{ \rm subject to } \;\; z \in XP_\gamma
$$
has a unique solution $\hat z \in  XP_\gamma$, namely, the projection
of $y$ onto $XP_\gamma$. Consequently, $\hat z = X\hat b$ for some
$\hat b \in P_\gamma$, where $\hat b$ is not necessarily unique.
Finally, $\hat b$ clearly is a solution of the optimization problem
\eqref{eq:constrained}.
\end{proof}

Before we turn to the proof of Proposition~\ref{prop:existence}, we
make the following observations. Note that we can decompose the
polyhedron $P_\gamma = \{b \in \R^p: \pen(b) \leq \gamma\} = \{b \in
\R^p : u_1'b,\dots,u_l'b \leq \gamma\}$, where $\gamma \geq 0$, into
the sum of a polyhedral cone (the so-called recession cone of
$P_\gamma$) and a polytope, \citep[see, e.g.,][Theorem~1.2 and
Proposition~1.12]{Ziegler12}. For $\gamma = 1$, we can therefore write
$$
P_1 = \{b \in \R^p: u_1'b \leq 0,\dots,u_l'b \leq 0\} + E,
$$
where $E$ is a polytope and therefore bounded. For arbitrary $\gamma
\geq 0$, we then write
\begin{equation} \label{eq:Pgamma}
P_\gamma = P_0 + \gamma E.
\end{equation}

\begin{proof}[Proof of Proposition~\ref{prop:existence}]
Let $(\beta_m)_{m \in N}$ be a minimizing sequence of $f$. By
Lemma~\ref{lem:min_sequence}, both sequences $(X\beta_m)_{m \in \N}$
and $(\pen(\beta_m))_{m \in \N}$ converge to, say, $l$ and $\gamma$,
respectively. This implies that
$$
\frac{1}{2}\|y - l\|_2^2 + \gamma = \inf_{b \in \R^p} f(b).
$$ 
Let $\hat\beta$ be an arbitrary solution of \eqref{eq:constrained}. We
prove that $f(\hat\beta) = \|y - l\|^2_2 + \gamma$. For this, we
distinguish the following two cases.

1) Assume that $\gamma > 0$. For $n$ large enough so that
$\pen(\beta_m) > 0$, we set $u_n$ as
$$
u_m = \frac{\gamma}{\pen(\beta_m)}\beta_m.
$$
Clearly, $\pen(u_m) = \gamma$ so that $u_m \in P_\gamma$.
Consequently, by definition of $\hat\beta$, we have $\|y - X\hat\beta\|_2^2 \leq
\|y - Xu_m\|_2^2$ and $\pen(\hat\beta) \leq \gamma$, so that
$$
f(\hat\beta) = \frac{1}{2} \left\|y - X\hat\beta\right\|^2_2 +
\pen(\hat\beta) \leq \frac{1}{2}\left\|y - Xu_m\right\|^2_2 + \gamma
\longrightarrow \frac{1}{2}\|y-l\|^2_2 + \gamma
$$
as $m \to \infty$, implying $f(\hat\beta) = \inf\{f(b) : b \in \R^p\}$.
 
2) Assume that $\gamma = 0$. Using \eqref{eq:Pgamma}, we can write
$\beta_m = u_m + \pen(\beta_m) v_m$ with $u_m \in P_0$ and $v_m \in
E$, where $E$ is bounded. Since $X\beta_m \to l$ and $\pen(\beta_m)v_n
\to 0$ one may deduce that also $Xu_m \to l$, yielding
$$
f(\hat\beta) = \frac{1}{2}\left\|y - X\hat\beta\right\|^2_2 \leq
\frac{1}{2}\left\|y - Xu_m\right\|^2_2
\longrightarrow \frac{1}{2}\left\|y - l\right\|^2_2$$
as $m \to \infty$ implying again that $f(\hat\beta) = \inf\{f(b) : b
\in \R^p\}$ which completes the proof.
\end{proof}

\subsection{A characterization of the noiseless recovery condition for the supremum norm} 
\label{subapp:nrc-sup}

Note that the noiseless recovery condition is always satisfied for
$\beta = 0$. We give a characterization for $\beta \neq 0$ when the
penalty term is given by the supremum norm.

\begin{proposition} 
\label{prop:nrc-sup} Let $X \in \R^{n \times p}$
and $\beta \in \R^p$ where $\beta \neq 0$ and $I = \{j \in [p]:
|\beta_j| < \|\beta\|_\infty\}$. Furthermore, let $\tilde X = (\tilde
X_1|X_{ I})$ where
$$
\tilde X_1 = X_{I^c}\sign(\beta_{I^c}).
$$
Then
$$
\exists  \lambda > 0, \exists \hat\beta \in \SsupL(X\beta) \text{ \rm with } 
\hat\beta \supequiv \beta  \iff  
e_1 \in \row(\tilde X) \text{ \rm and } \| X'(\tilde X')^+ e_1\|_1 \leq 1, 
$$
where $e_1 = (1,0,\dots,0)' \in \R^p $. 
\end{proposition}
Before presenting the proof, recall that the subdifferential of the
$\ell_\infty$-norm at $0$ is the unit ball of the $\ell_1$-norm, and
for $\beta \neq 0$, this subdifferential is equal to
\begin{align}
\dsup(\beta) \nonumber & =
\left\{s \in \R^p:\|s\|_1 \leq  1 \text{ \rm and } s'\beta = \|\beta\|_\infty \right\} \\
\label{eq:sudiff_sup} & = \left\{s \in \R^p: \|s\|_1 = 1 \text{ \rm and } \forall j \in [p] \; 
\begin{cases} s_j \beta_j \geq 0 & \text{ \rm if } |\beta_j| = \|\beta\|_\infty \\
s_j = 0 & \text{ \rm otherwise} \end{cases}\right\}.
\end{align}

\begin{proof}
$(\implies)$ Assume there exists $\lambda > 0$ and $\hat\beta \in
\SsupL(X\beta)$ such that $\hat\beta \supequiv \beta$. Then
\begin{equation}
\label{eq:subdiff_supnorm} \frac{1}{\lambda} X'(X\beta - X\hat\beta) \in
\dsup(\hat\beta) = \dsup(\beta).
\end{equation}
We set $c = (\|\beta\|_\infty,\beta_I')'$ and $\hat c =
(\|\hat\beta\|_\infty,\hat\beta_I')'$. By construction, $\tilde Xc =
X\beta$. Moreover, since $\dsup(\beta) = \dsup(\hat\beta)$, we also
have $\tilde X\hat c = X\hat\beta$. Consequently, by
\eqref{eq:subdiff_supnorm}, we get
$$
\frac{1}{\lambda} X'\tilde X(c - \hat c) \in \dsup(\beta).
$$
Therefore, using \eqref{eq:sudiff_sup}, we get that 
$$
X_I'X(c - \hat c) = 0,
$$
as well as
$$
\beta'\frac{1}{\lambda} X'\tilde X(c - \hat c) 
= \frac{1}{\lambda}\beta_{I^c}'X_{I^c}'\tilde X(c - \hat c) 
= \frac{1}{\lambda}\|\beta\|_\infty\sign(\beta_{I^c})'X_{I^c}'\tilde X(c - \hat c) 
= \frac{1}{\lambda}\|\beta\|_\infty \tilde X_1'\tilde X(c - \hat c) 
= \|\beta\|_\infty,
$$
so that 
$$
\frac{1}{\lambda} \tilde X_1'\tilde X(c - \hat c)  = 1.
$$
Therefore, we may conclude
$$
\frac{1}{\lambda}\tilde X'\tilde X(c - \hat c) = e_1 \implies 
\tilde X(c-\hat c) = \lambda (\tilde X')^+ e_1,
$$
which also yields 
$$
\frac{1}{\lambda}X'(X\beta - X\hat\beta) = \frac{1}{\lambda} X'\tilde X(c - \hat c) =
X'(\tilde X')^+e_1 \in \dsup(\beta).
$$
We therefore immediately get $\|X'(\tilde X')^+e_1\|_1 \leq 1$. It
remains to show that and $e_1 \in \row(\tilde X)$. Note that. analogously to above,
$X'(\tilde X')^+e_1 \in \dsup(\beta)$ implies that
$$
\tilde X'(\tilde X')^+e_1 = e_1.
$$
Since $\tilde X'(\tilde X')^+$ is the orthogonal projection onto $\row(\tilde X)$, 
we may deduce that $e_1 \in \row(\tilde X)$.

\medskip

$(\impliedby)$ As above, let $c = (\|\beta\|_\infty,\beta_I')'$ and
set $\hat c = c - \lambda \tilde X^+(\tilde X')^+e_1 = (\hat c_1,\hat
c_{-1})'$, where the first component is $\hat c_1$ and remaining
components are $\hat c_{-1}$. We define $\hat\beta$ through
$$
\hat\beta_{I^c} = \hat c_1 \sign(\beta_{I^c}) \text{ \rm and } \hat\beta_I = \hat c_{-1}.
$$
Since $c_1 = \|\beta\|_\infty$ for small enough $\lambda > 0$, we have
$\{j \in [p] : |\hat\beta_j| < \|\hat\beta\|_\infty\} = \{j \in [p] :
|\beta_j| < \|\beta\|_\infty\} = I$ as well as $\beta_j \hat\beta_j >
0$ for $j \notin I$. Therefore, for small enough $\lambda$,
$\dsup(\hat\beta) = \dsup(\beta)$ holds. To conclude the proof, it
suffices to show that $\hat\beta \in \SsupL(X\beta)$, i.e.,
$\frac{1}{\lambda}X'(X\beta - X\hat\beta) \in \dsup(\hat\beta)$. Since
$\col((\tilde X')^+) = \col(\tilde X)$ and $\tilde X\tilde X^+$ is the
orthogonal projection onto $\col(\tilde X)$, we get
$$
\frac{1}{\lambda}X'(X\beta - X\hat\beta) = \frac{1}{\lambda}X'(\tilde Xc - \tilde X\hat c) = 
X'\tilde X\tilde X^+(\tilde X')^+e_1 = X'(\tilde X')^+ e_1,
$$
so that left to show is $X'(\tilde X')^+ e_1 \in \dsup(\beta)$, which
holds if both $\|X'(\tilde X')^+e_1\|_1 \leq 1$ and $\beta'X'(\tilde
X')^+e_1 = \|\beta\|_\infty$ are true. The first inequality holds by
assumption. To show the latter, note that the assumption $\tilde
X'(\tilde X')^+ e_1 = e_1$ implies that
$$
\|\beta\|_\infty = c'e_1 = c'\tilde X'(\tilde X')^+ e_1 =
\beta'X'(\tilde X')^+e_1.
$$
Consequently, for $\lambda > 0$ small enough, $\hat\beta \in
\SsupL(X\beta)$.
\end{proof}




\end{document}